\documentclass[oneside,10pt]{amsart}

\usepackage{geometry} 
\usepackage{graphicx}
\usepackage{amsmath} 
\usepackage{amsthm} 
\usepackage{color} 
\usepackage{overpic} 
\usepackage{contour} 
\contourlength{1pt}
\usepackage[bookmarksnumbered=true]{hyperref} 
\hypersetup{
     colorlinks = true,
     linkcolor = blue,
     anchorcolor = blue,
     citecolor = teal,
     filecolor = blue,
     urlcolor = blue
     }
\usepackage{tikz} 

\def\centerarc[#1](#2)(#3:#4:#5);%
{

    \draw[#1]([shift=(#3:#5)]#2) arc (#3:#4:#5);
}

\newcommand{\Gr}{{\rm Gr}}
\newcommand{\RR}{\mathbb{R}}

\newcommand{\nn}{\mathbf{n}}
\newcommand{\lk}{{\rm lk}}
\newcommand{\sign}{{\rm sign}}
\newcommand{\wind}{{\rm wind}}
\newcommand{\dist}{{\rm dist}}
\newcommand{\myringerl}[1]{\fboxsep1pt\fbox{#1}}
\newcommand{\mm}{\phantom{-}}

\definecolor{myblue}{rgb}{0.11,.19,.8}
\definecolor{myred}{rgb}{0.8,.014,.04}

\newtheorem{theorem}{Theorem}[section]
\newtheorem{proposition}[theorem]{Proposition}

\theoremstyle{definition}
\newtheorem{definition}[theorem]{Definition}
\newtheorem{example}[theorem]{Example}
\newtheorem{remark}[theorem]{Remark}

\makeatletter\def\@captionfont{\normalfont\footnotesize}\makeatother

\newlength{\myparindent}
\title{A differential approach to Maxwell-Cremona liftings}

\author{Oleg Karpenkov, 
Fatemeh Mohammadi, 
Christian M{\"u}ller, and 
Bernd Schulze}

\date{\today}

\begin{document}

\maketitle

\begin{abstract}
  In 1864, J. C. Maxwell 
  introduced a link between self-stressed frameworks in
  the plane and piecewise linear liftings to 3-space. This connection
  has found numerous applications in areas such as discrete geometry,
  control theory and structural engineering. While there are some
  generalisations of this theory to liftings of $d$-complexes in
  $d$-space, extensions for liftings of frameworks in $d$-space for $d\geq
  3$ have been missing.
  In this paper, we introduce and study differential liftings on general graphs using
  differential forms associated with the elements of the homotopy groups of the complements to the frameworks. Such liftings 
 play the role of integrands for the classical notion of liftings for planar
  frameworks.
  We show that these differential liftings have a natural extension to 
     self-stressed frameworks in higher dimensions.
 As a result we generalise the notion of classical liftings to both graphs and multidimensional $k$-complexes in  $d$-space ($k=2,\ldots, d$).
  Finally we discuss a natural representation of generalised liftings as real-valued functions on Grassmannians.  
\end{abstract}

{\hypersetup{linkcolor=black}
{\tableofcontents}}

\section*{Introduction}

In 1864, J.~C.~Maxwell and later also L.~Cremona showed that a planar framework in $\RR^2$ with a self-stress has a dual reciprocal diagram and can be lifted vertically to a polyhedral
surface in $\RR^3$.
The present paper is dedicated to answer the following question:
\emph{What are the analogues of liftings for self-stressed frameworks in
$\RR^n$?} For $n\geq 3$, the answer to this question was not known. In
order to approach this question, we construct a new natural extension of
liftings for general frameworks in $\RR^n$ and for Rybnikov's polyhedral self-stressed complexes. This will provide the answer to the above problem.

The story begins with two papers by J.~C.~Maxwell~\cite{maxwell1864xlv,maxwell1870reciprocal} and one by L.~Cremona~\cite{cremona1872figure} in which the authors founded the field of Graphic Statics. They discovered that self-stressability of planar frameworks can be visualised by piecewise linear functions whose singular set is the framework itself. This theory is closely related to the duality between Voronoi diagrams and Delaunay tessellations.    The result of Maxwell and Cremona has important connections and applications in  rigidity theory, polyhedral combinatorics and discrete and computational geometry; see e.g.~\cite{WW82,Au87,Sugi,CW93,crapo1994spaces,hopcroft1992paradigm,streinu,streinuerr,RRS}. In particular, the technique of liftings has contributed to the solution of the Carpenter's Rule problem \cite{Con03}. Liftings are also used in a well known proof of Steinitz's Theorem from polyhedral combinatorics~\cite{richter2006realization}. (See also~\cite{EadesGarvan} and~\cite{Pak}.)
Recently, F.~Mohammadi and X.~Wu in~\cite{mohammadi2022rational} introduced a link between self-stresses and Chow rings on toric varieties related to convex polytopes.
There has also been a resurgence of interest in techniques from Graphic Statics in architecture and structural engineering \cite{mbmm}. (In  engineering, the 3-dimensional polyhedra projecting down to self-stressed frameworks in the plane are commonly known as \emph{discrete Airy stress function polyhedra}.) 
Moreover, liftings have found modern  applications in control theory \cite{nguyen2017convex, nguyen2017constructive}. For general references about liftings in discrete and computational geometry, we refer the reader to
\cite[Chapter 61]{handbookDCG}, as well as to  \cite{WAC} and  \cite{sitharam2018handbook}. 

Let us mention what is currently known about generalisations of the  Maxwell-Cremona correspondence. 
In 1993, H. Crapo and W. Whiteley investigated an extension of this correspondence to $3$-space in an unpublished manuscript \cite{crapo1993}. 3D Graphic Statics has recently been further investigated in the engineering community (see e.g. \cite{marina}).
In 1999, K.~Rybnikov introduced self-stresses on polyhedral $n$-complexes in $\RR^n$ and studied liftings of these objects~\cite{Rybnikov991,Rybnikov99}. More recently, 
A.~Rib\'o~Mor discussed liftings for planar self-touching frameworks \cite{ribo}. In 2014, C.~Borcea and I.~Streinu proved an analogue of the theorem of Maxwell and Cremona for planar periodic frameworks \cite{borcea2014liftings,BorStr}. Finally, J.~Erickson and P.~Lin studied reciprocal diagrams of toroidal frameworks in \cite{ErLin}, and Z.~Cooperband, R.~Ghrist and J.~Hansen very recently gave a description of Graphic Statics for planar and non-planar frameworks using cosheaf homology  in \cite{cooper}.

Extensions of liftings for frameworks in $\RR^n$ for $n\geq 3$ have not yet been obtained. In this paper we give a construction for this case and for more general cases. As was already done by K.~Rybnikov in the literature~\cite{Rybnikov991,Rybnikov99},
the liftings of $n$-complexes in $\RR^n$ are defined on the complement to $n$-complexes.
Our idea is to consider these complements as the set of path-connected
components denoted as $\pi_0$.
This consideration gives rise to a natural generalisation of liftings to the cases of complexes of arbitrary dimensions (in particular for frameworks in $\mathbb R^n$) and to the case of non-planar frameworks.
Here the elements of the higher fundamental groups $\pi_m$ of the complement to the frameworks/complexes will serve as the domain sets in the definition of `differential liftings'.

In this paper, we consider piecewise constant $m$-forms on the elements of $\pi_m$ of the complement to the frameworks/complexes in $\mathbb R^n$ which we call {\it differential liftings}
(see Definitions~\ref{def:diflift}, \ref{def:difflifting-n}, and~\ref{def:diflfting-polyhedra}).
Differential liftings satisfy a certain boundary condition at every vertex induced by equilibrium conditions on the stress coefficients on edges incident with this vertex.
In the classical case  the liftings of planar frameworks in $\mathbb R^2$ 
are the integral  graphs of the corresponding differential 1-forms.
We show that for every framework in $\mathbb R^n$ with a self-stress there exists a
unique differential lifting (Theorems~\ref{thm:well} and ~\ref{thm:well2}).
This provides the answer to the original question regarding generalisations of liftings for self-stressed frameworks in $\mathbb R^n$.
Further we extend this result to the case of general polyhedral complexes (Theorem~\ref{thm:well3}).
We briefly discuss generalisations of liftings in the last section of this 
paper.
It is interesting to note that
such liftings can be represented 
as functions on affine Grassmannians
(see Subsection~\ref{Grassmannian liftings}).

We work mostly with differential liftings rather than with usual liftings for the following reasons.
On the one hand the differential liftings are simpler than usual liftings in the sense that they are piecewise constant over $\pi_0$ rather than piecewise linear.
So they are invariants of the elements of $\pi_0$. Therefore, we prefer to use  them when extending differential liftings to $\pi_m$ as invariants of its elements. 
On the other hand, the differential liftings provide the same information on self-stressability as the classical liftings for planar frameworks.
Indeed, the liftings are obtained from differential liftings by integration of the latter.

{\bf This paper is organised as follows.}
We begin by providing the necessary classical notions from Geometric Rigidity Theory and Graphic Statics in Section~\ref{Preliminaries}. In Section~\ref{sec:2_dim}, we approach the notion of liftings of self-stressed frameworks from a new perspective using differential forms and introduce the natural notion of differential liftings. 
In Section~\ref{sec:n_dim}, we define this notion for frameworks in $n$-dimensional space. Finally, in   Section~\ref{sec:polytopal_embedding} we extend the concept of  differential liftings to the case of polyhedral complexes introduced by K.~Rybnikov in 1999.
We conclude the paper in Section~\ref{Definition of liftings in general settings}
with discussions on how to extend standard liftings from differential
liftings in general settings.

\section{Preliminaries}
\label{Preliminaries}

Our paper is based on two fundamental notions which are self-stressed
frameworks (Definition~\ref{def:framework}) and liftings
(Definition~\ref{def:liftings}). Self-stressed frameworks are realisations
of graphs in $\RR^n$ which are in static equilibrium. Liftings are,
broadly speaking, polyhedra in a higher dimensional space whose vertical
projections are given self-stressed frameworks.

\begin{definition}
  \label{def:framework}
  Let $G = (V, E)$ be an arbitrary finite graph without loops and without
  multiple edges where $V = \{v_1, \ldots, v_d\}$ denotes the set of
  \emph{vertices} and $E$ the set of \emph{edges}.
  We denote by $v_i v_j$ the edge joining $v_i$ and $v_j$. 
  \begin{itemize}
    \item A \emph{framework} $G(p)$ in $\RR^n$ consists of a graph
      $G = (V, E)$ and an injective map $p : V \to \RR^n$. We often denote
      $p(v_i)$ by $p_i$.
      We say that there is an \emph{edge} between $p_i$ and $p_j$ if 
      $v_i v_j$ is an edge of $G$ and denote it by $p_i p_j$. 
      Note that we require that all points are pairwise distinct.
    \item A \emph{stress} $\omega$ on a framework is an assignment of real
      scalars $\omega_{ij}$ to its edges $v_i v_j$ which are called
      \emph{stress coefficients} (also known as \emph{tensions} and
      \emph{compressions} depending on the sign) with the property
      $\omega_{ij} = \omega_{ji}$. 
      We also set $\omega_{ij} = 0$ if there is no edge between the
      corresponding vertices.
    \item A stress $\omega$ is called a \emph{self-stress} if the
      equilibrium condition 
      \begin{equation*}
        \sum\limits_{\{j \mid j \neq i\}} \omega_{ij} (p_i - p_j) = 0
      \end{equation*}
      is fulfilled at every vertex $i$.
      By $p_i - p_j$ we denote the vector from the point $p_j$ to the
      point $p_i$.  
    \item A pair $(G(p), w)$ is called a \emph{tensegrity} if $\omega$ is
      a self-stress  for the framework $G(p)$.
    \item A tensegrity $(G(p), w)$ (or stress $\omega$) is said to be
      \emph{non-zero} if there exists an edge $v_i v_j$ of the framework
      that has a non-vanishing stress coefficient $\omega_{ij} \neq 0$.
  \end{itemize}
\end{definition}

We recall in the following definition the classical notion of a polyhedral
lifting over a planar graph that has been put into relation with
self-stressed frameworks by Maxwell~\cite{maxwell1870reciprocal}.

\begin{figure}[t]
  \begin{overpic}[width=.30\textwidth]{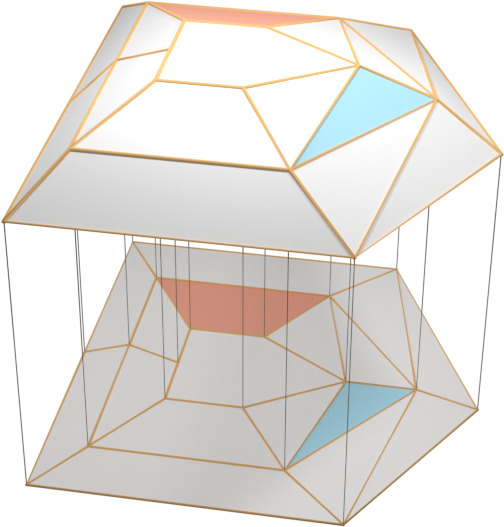}
  \end{overpic}
  \caption{Polyhedral lifting of a planar framework (see
  Definition~\ref{def:liftings}).}
  \label{fig:lifting}
\end{figure}

\begin{definition}
  \label{def:liftings}
  Let $G$ be a planar graph and $G(p)$ a corresponding framework embedded
  as planar straight-line graph in $\RR^2$.
  Denote by $F$ the set of connected components of $\RR^2\setminus G(P)$;
  we call the elements of $F$ the {\it faces}. The only unbounded face is
  called the {\it infinite face}.
  \\
  A continuous and piecewise affine function $L: \RR^2 \to \RR$ is called
  \emph{$($polyhedral$)$ lifting} of $G(p)$ if $L$ restricted to any face of
  $F$ is an affine map. See Figure~\ref{fig:lifting} for an illustration
  of the graph of such a function $L$.
\end{definition}
Note that all faces are open sets with a polygonal boundary.

The graph $(u, v, L(u, v))$ of a polyhedral lifting is therefore a
polyhedral surface whose edges and vertices vertically project to the
given graph.

\begin{theorem}[\cite{maxwell1870reciprocal}]
  \label{thm:maxwell}
  Let $G$ be a planar graph and $G(p)$ a framework embedded in $\RR^2$
  without crossing edges. Then  $G(p)$ has a non-zero self-stress if and
  only if $G(p)$ has a non-trivial polyhedral lifting $L$.
\end{theorem}

\noindent
\begin{minipage}[b][][t]{.71\textwidth}
  \begin{remark}
    %\label{rem:reciprocal}
    Let us mention the classical notion of the Maxwell reciprocal
    diagram. The \emph{Maxwell reciprocal diagram} of a framework $G(p)$
    in $\RR^2$ is a framework $G^*(p^*)$ defined on the combinatorial dual
    graph $G^*$ such that corresponding edges of the frameworks $G$ and
    $G^*$ are orthogonal to each other with lengths given by the
    corresponding stress coefficients. See the figure aside for an
    illustration. It is important to note that the existence of the
    Maxwell reciprocal diagram $G^*(p^*)$ is equivalent to non-zero
    self-stressability of $G(p)$.
  \end{remark}
\end{minipage}
\hfill\hfill
\begin{minipage}[b][][t]{.22\textwidth}
  \begin{overpic}[width=\textwidth]{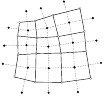}
  \end{overpic}
\end{minipage}
\hfill{}

For $G(p)$ with self-stress $\omega$, let us denote the normal vector to the face of the
polyhedral lifting $L$ corresponding to a face $f\in F$  by $\nu_f$ and let these vectors be ``normalized''
to $\nu_f = (*, *, -1)$. Then from~\cite[Eqn.~(4)]{BorStr} we know that 
\begin{equation}
  \label{eq:normals}
  \nu_{f_2} - \nu_{f_1} = \omega_{ij} (p_j - p_i)^\perp
\end{equation}
for neighbouring faces $f_1, f_2 \in F$ with common edge $p_i p_j$.

\section{Differential liftings in $\RR^2$}
\label{sec:2_dim}

In this section, we generalise the concept of the classical polyhedral 
lifting (Definition~\ref{def:liftings}). By Maxwell's theorem
(Theorem~\ref{thm:maxwell}) the lifting exists exactly for self-stressed
frameworks over their embedding as planar straight-line graph.
Instead of considering the lifting itself we now switch to a
discrete differential of that lifting. It turns out that such a discrete
differential (which we call \emph{differential lifting}, cf.\
Definition~\ref{def:diflift}) exists also for frameworks which are not
embedded as planar graphs as well as for frameworks over non-planar
graphs.

While an embedding of a planar graph constitutes a cell decomposition of
$\RR^2$ into faces, a general realisation of a graph in $\RR^2$ might have
crossing edges and divides the plane into cells which we call
\emph{chambers}.

\begin{definition}
  For a framework $G(p)$, a \emph{chamber} is a connected component of the
  complement of $G(p)$ to $\RR^2$. Two different chambers are neighbours
  if they share an edge (see Figure~\ref{fig:chambers} left) or if they
  share a non-empty open sub-segment of an edge (see
  Figure~\ref{fig:chambers} right). 
\end{definition}

Note that as the framework $G(p)$ is finite, there is always an unbounded
chamber $C_\infty$, the one ``around'' the framework.

A key idea in our paper is to relate differential $1$-forms to the chambers
of the framework. For that we identify vectors $v = (v_1, v_2)$ with
differential $1$-forms $dv = v_1\, dx + v_2\, dy$ via linear combinations
of coordinate differentials $dx$ and $dy$. The space of differential
$1$-forms is denoted by $\Omega^1(\RR^2)$. 
For a given framework $G(p)$, we will consider maps from the set of
chambers to the space of constant $1$-forms on $\RR^2$. The space of
chambers is canonically identified with the elements of
$\pi_0(\RR^2 \setminus G(p))$, the set of path components. Once the map
from the set of chambers to $\pi_0(\RR^2 \setminus G(p))$ is fixed, we can
say that the corresponding $1$-forms are \emph{associated} with the
chambers. 

\begin{definition}
  \label{def:nei1}
  Consider a framework $G(p)$ with a self-stress $\omega$. Let $C_1$ and
  $C_2$ be two neighbouring chambers of $G(p)$ along $p_i p_j$ (i.e.\ they are
  either sharing the same edge $p_i p_j$ or a non-empty open sub-segment of it).
  The stress coefficient of $p_i p_j$ is $\omega_{ij}$. Let $\alpha_1$ and
  $\alpha_2$ be two constant $1$-forms associated with $C_1$ and $C_2$,
  respectively. We say that $\alpha_1$ and $\alpha_2$ satisfy the
  \emph{neighbouring condition} if 
  \begin{equation}
    \label{eq:nei1}
    \alpha_2 - \alpha_1 
    =
    \sign(\det(\nn_{12}, p_i - p_j)) \cdot \omega_{ij} \cdot d(p_i{-}p_j),
  \end{equation}
  where $\nn_{12}$ is the unit normal vector to $p_i - p_j$ pointing in
  the direction from $C_1$ to $C_2$ (cf.\ Figure~\ref{fig:chambers}). 
\end{definition}

Note that the neighbouring condition is independent of the order of
$\alpha_1$ and $\alpha_2$.

\begin{figure}[h]
  \hfill
  \begin{overpic}[width=.38\textwidth]{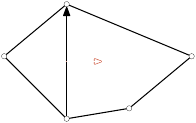}
    \put(27,61){$p_i$}
    \put(27,-1){$p_j$}
    \put(17,30){$C_1$}
    \put(67,30){$C_2$}
    \put(5,5){$C_\infty$}
    \put(39,35){$\nn_{12}$}
  \end{overpic}
  \hfill
  \begin{overpic}[width=.38\textwidth]{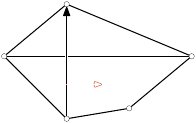}
    \put(27,61){$p_i$}
    \put(27,-1){$p_j$}
    \put(22,22){$C_1$}
    \put(64,22){$C_2$}
    \put(48,42){$C_3$}
    \put(22,42){$C_4$}
    \put(5,5){$C_\infty$}
    \put(39,23){$\nn_{12}$}
  \end{overpic}
  \hfill{}
  \caption{Two different general frameworks $G(p)$ and their
  chambers. Note that two chambers are neighbours if they share an edge
  or a non-empty open sub-segment of an edge.
  \emph{Left:} The framework consists of two neighbouring faces, a triangle
  and a quadrilateral. The two neighbouring chambers $C_1$ and $C_2$ share
  the common edge $p_i p_j$. 
  \emph{Right:} The framework has two crossing edges which do not
  intersect in a vertex of the framework. The pentagon is therefore
  divided into four chambers. Two neighbouring chambers, e.g., $C_1$ and
  $C_2$, share a non-empty open sub-segment of the edge $p_i p_j$. 
  }
\label{fig:chambers}
\end{figure}
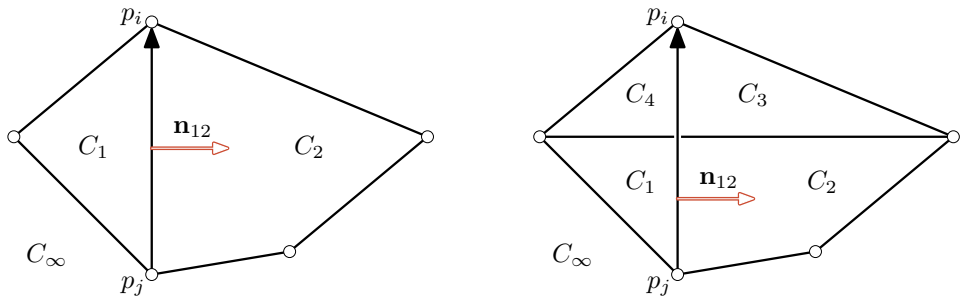

The following definition of differential liftings is central for the
present paper. It generalises the concept of the classical polyhedral
lifting (Definition~\ref{def:liftings}), in a sense, to a discrete differential of
a lifting which may or may not exist, as explained below.

\begin{definition}\label{diff-liff}
  \label{def:diflift}
  For the framework $G(p)$ and a self-stress $\omega$ of $G(p)$, we call
  a mapping 
  \begin{equation*}
    l_{\omega, p} : 
    \pi_0(\RR^2 \setminus G(p)) \longrightarrow \Omega^1(\RR^2)
  \end{equation*}  
  a \emph{differential lifting} if it satisfies the \emph{neighbouring
  condition} for each pair of neighbouring chambers and if it takes zero
  on the unbounded chamber of $C_\infty \in \pi_0(\RR^2 \setminus G(p))$.
\end{definition}

In Theorem~\ref{thm:well}, we show that a differential lifting as defined
in Definition~\ref{def:diflift} exists and that it is unique.
Note that the set of differential liftings over a fixed framework $G(p)$
is a vector space with
\begin{equation*}
  l_{\omega^1, p} + l_{\omega^2, p} = l_{\omega^1 + \omega^2, p}
  \qquad\text{and}\qquad
  \lambda l_{\omega, p} = l_{\lambda \omega, p}.
\end{equation*}

\begin{theorem}
  \label{thm:well}
  For every framework $G(p)$ with a self-stress $\omega$, there exists a
  unique differential lifting $l_{\omega, p}$.  
\end{theorem}
\begin{proof}
  Let us first assume that our framework $G(p)$ has no \emph{crossing
  edges}.

  By Definition~\ref{diff-liff}, the value of the unbounded chamber
  $C_\infty$ is set to be zero. 
  Any other value of $l_{\omega, p}$ can be obtained by adding
  differential forms via the neighbouring condition,
  Equation~\eqref{eq:nei1}, along a path of chambers starting at
  $C_\infty$, where a \emph{path} between two chambers is a sequence of
  chambers that share a common edge.
  Our aim is to show that different paths with the same starting and
  ending chambers do not give different values at the last chamber.

  Consider a path $\mathcal{C}$ of chambers $C_1, \ldots, C_k$ with
  associated differential forms denoted by $\alpha_1, \ldots, \alpha_k$.
  Adding differential forms along the path via Equation~\eqref{eq:nei1}
  yields
  $$
  \alpha_k = \alpha_1 + \alpha_{\mathcal{C}},
  \qquad \hbox{where} \qquad  
  \alpha_{\mathcal{C}} = \sum\limits_{r = 2}^k
  (\alpha_r - \alpha_{r - 1}),
  $$
  where $(\alpha_r - \alpha_{r - 1})$ stands for the right hand side of
  Equation~\eqref{eq:nei1}.

  For the reverse path $C_k, \ldots, C_1$, the differential forms then
  satisfy $\alpha_1 = \alpha_k - \alpha_{\mathcal{C}}$.

  Let us consider two paths $\mathcal{C}_1$ and $\mathcal{C}_2$
  with the same starting chamber $C_s$ and the same end chamber $C_e$.
  We will show that these two (potentially different) paths give the same
  value at the last chamber. Adding differential forms along the two paths
  yields 
  \begin{equation*}
    \alpha^1_e = \alpha_s + \alpha_{\mathcal{C}_1}
    \qquad\text{and}\qquad
    \alpha^2_e = \alpha_s + \alpha_{\mathcal{C}_2}.
  \end{equation*}
  Consequently, adding differential forms along the path
  $\mathcal{C}_1$ followed by the reversed path $\mathcal{C}_2$ 
  (which combined is a \emph{closed} path as the starting chamber is the
  same as the end chamber) adds 
  $\alpha_{\mathcal{C}_1} - \alpha_{\mathcal{C}_2}$ to $\alpha_s$.
  Therefore, the two paths $\mathcal{C}_1$ and $\mathcal{C}_2$ imply equal
  values $\alpha^1_e = \alpha^2_e$ at $C_e$ if and only if the closed path
  above just adds zero to $\alpha_s$.

  The uniqueness of the definition of $l_{\omega, p}$ is therefore
  equivalent to showing that adding differential forms via Equation~\eqref{eq:nei1},
  along closed paths, adds zero to the starting chamber. We start with two
  elementary cycles before we cover a general closed path.

  \emph{Case 1}: Let us consider the closed path of chambers $C_1, C_2,
  C_3 = C_1$. This path corresponds to passing through an edge and going
  back. Let us assume $l_{\omega, p}(C_1) = \alpha_1$. Adding the values
  along this path yields first (cf.\ Figure~\ref{fig:chambers})
  \begin{equation*}
    \alpha_2 
    = 
    \alpha_1 
    + 
    \omega_{ij} \cdot \sign(\det(\nn_{12}, p_i - p_j)) \cdot d(p_i{-}p_j),
  \end{equation*}
  and further
  \begin{equation*}
    \alpha_3 
    = 
    \alpha_2 
    + 
    \omega_{ij} \cdot \sign(\det(\nn_{23}, p_i - p_j)) \cdot d(p_i{-}p_j).
  \end{equation*}
  since $\nn_{23} = \nn_{21} = -\nn_{12}$, inserting these values into the
  above formulas implies $\alpha_3 = \alpha_1$. Consequently nothing is
  added to the starting chamber.

  \emph{Case 2}: Let us consider a closed path around a point $p_i$ in
  counterclockwise direction, $C_1, \ldots, C_k = C_1$. Recursively
  inserting Equation~\eqref{eq:nei1} yields
  \begin{equation*}
  \alpha_k 
  =
  \alpha_1 
  + 
  \sum_{j = 1}^k \omega_{ij} \cdot \sign(\det(\nn_{j, j + 1}, p_i - p_j)) 
  \cdot d(p_i{-}p_j)
  =
  \alpha_1 + d\Big(\sum_{j = 1}^k \omega_{ij} (p_i - p_j)\Big) 
  = 
  \alpha_1,
  \end{equation*}
  where we take indices modulo $k$, i.e.\ $\nn_{k, k + 1} = \nn_{k, 1}$.
  The second equality holds true due to the linearity of the differential
  operator, and because all the signs in the sum are positive (because of
  counterclockwise orientation). The third equality holds by the equilibrium
  condition for self-stressed frameworks at $p_i$.
       
  \emph{General case}: Adding differential forms via~\eqref{eq:nei1} along
  any closed path adds zero to the starting chamber since it can be
  decomposed into elementary closed paths considered in Cases~1 and~2. 

  \noindent
  \begin{minipage}{.58\textwidth}
    \hspace*{\myparindent}
    Let us now consider the case where the framework $G(p)$ with
    self-stress $\omega$ has \emph{crossing edges} which we excluded so
    far. We can easily turn it into a self-stressed framework without
    crossing edges. First, we add all crossing points to the set of
    vertices and obtain a framework $G(\tilde p)$. Second, we modify the
    stress function $\omega$ to $\tilde \omega$ by the following rule:
    Consider two neighbouring vertices $\tilde p_k \tilde p_l$ forming an
    edge of $G(\tilde p)$ which lies on an ``old'' edge $p_i p_j$. Then
    the new stress coefficient $\tilde \omega_{kl}$ is defined by
    $$
    \tilde \omega_{kl} := \frac{\omega_{ij}}{\lambda_{kl}},
    $$
    where $\lambda_{kl}$ is the scaling factor between these two edges:
    $\|p_k - p_l\| = \lambda_{kl} \|p_i - p_j\|$.
    \end{minipage}
    \hfill
    \begin{minipage}{.36\textwidth}
      \begin{overpic}[width=\textwidth]{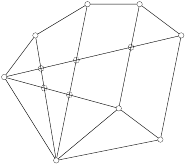}
        \put(-11,51){$p_i{=}\tilde p_i$}
        \put(96,73){$p_j$}
        \put(15,53){$\tilde p_k$}
        \put(35,58){$\tilde p_l$}
        \put(62,65){$\tilde p_m$}
        \put(29,45){$C_1$}
        \put(29,65){$C_2$}
      \end{overpic}
    \end{minipage}

    This is a self-stress since at every ``new'' vertex $\tilde p_l$
    opposite forces cancel: 
    $$
    \tilde\omega_{kl} (\tilde p_k - \tilde p_l) 
    + 
    \tilde\omega_{ml} (\tilde p_m - \tilde p_l)
    =
    \frac{\omega_{ij}}{\lambda_{kl}} \lambda_{kl} (p_i - p_j) 
    +
    \frac{\omega_{ij}}{\lambda_{ml}} \lambda_{ml} (p_j - p_i) 
    = 0.
    $$
    Analogously at ``old'' vertices $p_i$ the equilibrium condition still
    holds as any ``new'' edge $\tilde p_k \tilde p_i$ contributes the same
    force:
    $$
    \tilde\omega_{ki} (\tilde p_k - \tilde p_i) 
    =
    \frac{\omega_{ij}}{\lambda_{ki}} \lambda_{ki} (p_j - p_i) 
    =
    \omega_{ij} (p_j - p_i).
    $$

    Now note that the neighbouring condition~\eqref{eq:nei1}  along
    a ``new'' edge is the same as before when considered as two
    neighbouring chambers along a non-empty open sub-segment of an edge
    since
    \begin{equation*}
      \alpha_2 - \alpha_1 
      =
      \sign(\det(\nn_{12}, \tilde p_k - \tilde p_l)) \cdot 
      \tilde\omega_{kl} \cdot d(\tilde p_k{-}\tilde p_l)
      =
      \sign(\det(\nn_{12}, p_i - p_j)) \cdot 
      \omega_{ij} \cdot d(p_i{-}p_j),
    \end{equation*}
    which concludes the proof.
\end{proof}

\begin{figure}[ht]
  \hfill
  \begin{overpic}[width=.38\textwidth]{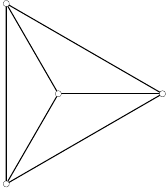}
    \put(33,47){$p_1$}
    \put(83,44){$p_2$}
    \put(-4,92){$p_3$}
    \put(-4,-2){$p_4$}
    \put(-4,50){\small$-1$}
    \put(34,81){\small$-1$}
    \put(34,15){\small$-1$}
    \put(20,71){\small$3$}
    \put(20,25){\small$3$}
    \put(57,46){\small$3$}
    \put(60,87){\small\myringerl{$0$}}
    \put(6,47){\small\myringerl{$-\sqrt{3}\, dy$}}
    \put(29,37){\small\myringerl{$\frac{3}{2} dx + \frac{\sqrt{3}}{2} dy$}}
    \put(29,58){\small\myringerl{$-\frac{3}{2} dx + \frac{\sqrt{3}}{2} dy$}}
    \put(47,87){$C_\infty$}
    \put(31,69){$C_1$}
    \put(11,60){$C_2$}
    \put(31,27){$C_3$}
  \end{overpic}
  \hfill
  \begin{overpic}[width=.40\textwidth]{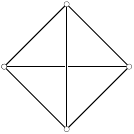}
    \put(96,55){$q_1$}
    \put(53,97){$q_2$}
    \put(-1,43){$q_3$}
    \put(54,0){$q_4$}
    \put(23,79){\small$-1$}
    \put(23,18){\small$-1$}
    \put(71,79){\small$-1$}
    \put(71,18){\small$-1$}
    \put(52,26){\small$1$}
    \put(30,52){\small$1$}
    \put(80,87){\small\myringerl{$0$}}
    \put(58,37){\small\myringerl{$dx + dy$}}
    \put(60,60){\small\myringerl{$-dx + dy$}}
    \put(20,60){\small\myringerl{$-dx - dy$}}
    \put(21,37){\small\myringerl{$dx - dy$}}
    \put(66,91){$C_\infty$}
    \put(59,71){$C_1$}
    \put(35,71){$C_2$}
    \put(35,24){$C_3$}
    \put(61,24){$C_4$}
  \end{overpic}
  \hfill{}
\caption{Two frameworks corresponding to the combinatorial graph $K_4$.
  With the vertices and stresses as detailed in Example~\ref{ex:k4} both
  frameworks are self-stressed. The chambers are denoted by $C_i$ and the
  associated $1$-forms are written with boxes. The \emph{left} framework
  $G(p)$ has a standard lifting $L$ in the sense of
  Theorem~\ref{thm:maxwell} as well as a differential lifting $l_{\omega,
  p}$. The framework $G(q)$ on the \emph{right} does not have a standard
  lifting $L$ as in Theorem~\ref{thm:maxwell}. It does, however,
  have a differential lifting $l_{\omega, q}$.
  }
  \label{fig:k4}
\end{figure}
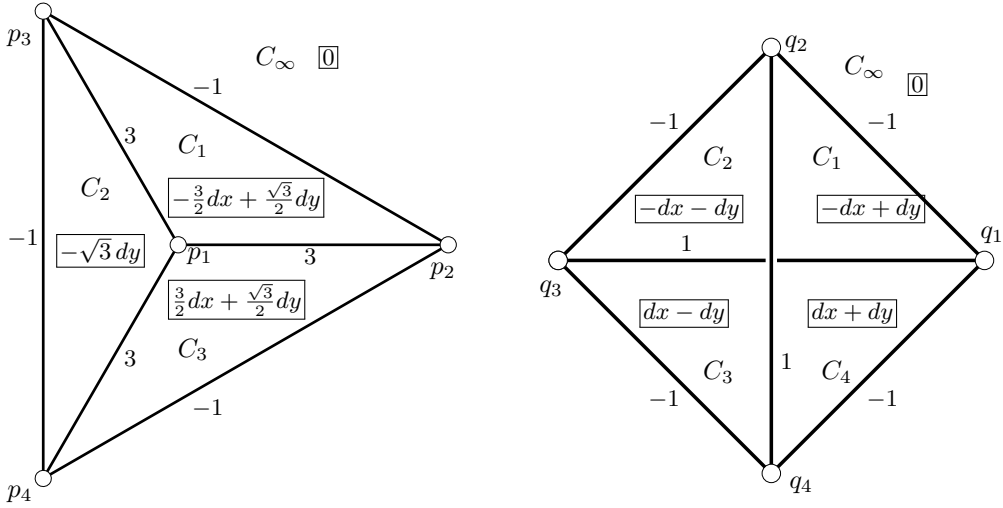

\begin{example}
  \label{ex:k4}
  Let us consider $K_4$, the smallest graph in the plane with a
  self-stress for all framework realisations.
  We consider two distinct frameworks $K_4(p)$ and $K_4(q)$ as depicted in 
  Figure~\ref{fig:k4}.
  Here the framework in Figure~\ref{fig:k4} (left) is planar whereas the
  framework in Figure~\ref{fig:k4} (right) is not.
  We have (cf.\ Figure~\ref{fig:k4} left) 
  \begin{equation*}
    \textstyle
    p_1 = (0, 0), 
    \quad
    p_2 = (1, 0), 
    \quad
    p_3 = (-\frac{1}{2}, \frac{\sqrt{3}}{2}), 
    \quad
    p_4 = (-\frac{1}{2}, -\frac{\sqrt{3}}{2}),
  \end{equation*}
  and stress coefficients $\omega_{12} = \omega_{13} = \omega_{14} = 3$, and 
  $\omega_{23} = \omega_{34} = \omega_{24} = -1$.
  The chambers are denoted by
  \begin{equation*}
    C_\infty = p_2 p_3 p_4,
    \quad
    C_1 = p_1 p_2 p_3,
    \quad
    C_2 = p_1 p_3 p_4,
    \quad
    C_3 = p_1 p_4 p_2.
  \end{equation*}
  The $1$-form associated to the unbounded chamber $C_\infty$ is per
  definition set to $0$. The associated $1$-forms corresponding to the
  other chambers are
   \begin{equation*}
    \textstyle
    l_{\omega, p}(C_1) = -\frac{3}{2} dx + \frac{\sqrt{3}}{2} dy,
    \quad
    l_{\omega, p}(C_2) = -\sqrt{3}\, dy,
    \quad
    l_{\omega, p}(C_3) = \frac{3}{2} dx + \frac{\sqrt{3}}{2} dy.
  \end{equation*}
  For the non-planar realisation of $K_4$ (Figure~\ref{fig:k4} right)
  we have
  \begin{equation*}
    q_1 = (1, 0), 
    \quad
    q_2 = (0, 1), 
    \quad
    q_3 = (-1, 0), 
    \quad
    q_4 = (0, -1),
  \end{equation*}
  where the edges $q_1 q_3$ and $q_2 q_4$ cross each other at the origin
  denoted by $O = (0, 0)$. The stress coefficients are 
  $\omega_{12} = \omega_{23} = \omega_{34} = \omega_{14} = -1$, and 
  $\omega_{13} = \omega_{24} = 1$.
  The chambers are denoted by 
  \begin{equation*}
    C_\infty = q_1 q_2 q_3 q_4,
    \quad
    C_1 = O q_1 q_2,
    \quad
    C_2 = O q_2 q_3, 
    \quad
    C_3 = O q_3 q_4, 
    \quad
    C_4 = O q_1 q_4.
  \end{equation*}
  The $1$-form associated to the unbounded chamber $C_\infty$ is per
  definition set to $0$. The associated $1$-forms corresponding to the
  other chambers are
  \begin{equation*}
    l_{\omega, q}(C_1) = dx - dy,
    \quad
    l_{\omega, q}(C_2) = -dx - dy,
    \quad
    l_{\omega, q}(C_3) = -dx + dy,
    \quad
    l_{\omega, q}(C_4) = dx + dy.
  \end{equation*}
  The framework $G(q)$ in Figure~\ref{fig:k4} (right) does not have a
  standard lifting as in Theorem~\ref{thm:maxwell}, but as we showed,
  it has a differential lifting $l_{\omega, q}$.
\end{example} 

\begin{example}
  \label{ex:prism}
  Another classical example is a framework with the combinatorics of a
  triangular prism as depicted in Figure~\ref{fig:prism}. The six vertices
  are
  \begin{equation*}
    p_1 = (-1, -1),
    \quad
    p_2 = (3, -1), 
    \quad
    p_3 = (3, 1), 
    \quad
    p_4 = (-1, 1), 
    \quad
    p_5 = (0, 0),
    \quad
    p_6 = (2, 0),
  \end{equation*}
  and the stress coefficients are
  $\omega_{12} = \omega_{34} = -1$, $\omega_{23} = \omega_{14} = -2$, 
  and $\omega_{15} = \omega_{45} = \omega_{26} = \omega_{36} = 4$.
  The chambers are denoted by 
  \begin{equation*}
    C_\infty = p_1 p_2 p_3 p_4,
    \quad
    C_1 = p_1 p_2 p_6 p_5,
    \quad
    C_2 = p_2 p_3 p_6, 
    \quad
    C_3 = p_3 p_4 p_5 p_6,
    \quad
    C_4 = p_1 p_4 p_5.
  \end{equation*}
  The $1$-form associated to the unbounded chamber $C_\infty$ is per
  definition set to $0$. The associated $1$-forms corresponding to the
  other chambers are
  \begin{equation*}
    l_{\omega, p}(C_1) = -4\, dx,
    \quad
    l_{\omega, p}(C_2) = 4\, dy,
    \quad
    l_{\omega, p}(C_3) = 4\, dx,
    \quad
    l_{\omega, p}(C_4) = -4\, dy.
  \end{equation*}
\end{example}

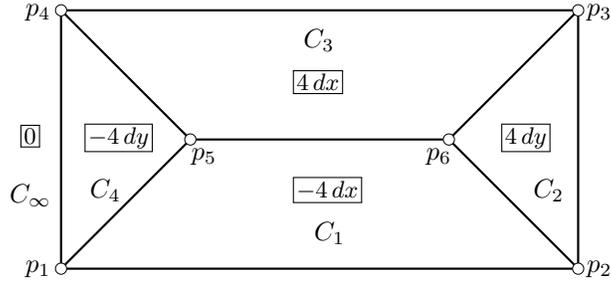
\begin{figure}[h]
  \hfill
  \begin{overpic}[width=.47\textwidth]{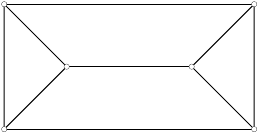}
    \put(-5,1){$p_1$}
    \put(100,1){$p_2$}
    \put(100,49){$p_3$}
    \put(-5,49){$p_4$}
    \put(26,22){$p_5$}
    \put(70,22){$p_6$}
    \put(-6,25){\small\myringerl{$0$}}
    \put(6,25){\small\myringerl{$-4\, dy$}}
    \put(84,25){\small\myringerl{$4\, dy$}}
    \put(45,15){\small\myringerl{$-4\, dx$}}
    \put(45,35){\small\myringerl{$4\, dx$}}
    \put(-8,14){$C_\infty$}
    \put(49,7){$C_1$}
    \put(90,15){$C_2$}
    \put(47,43){$C_3$}
    \put(7,15){$C_4$}
  \end{overpic}
  \hfill{}
  \caption{A framework corresponding to a combinatorial triangular prism. With the
  vertices and stress coefficients as detailed in Example~\ref{ex:prism}, the
  framework is self-stressed. The chambers are denoted by $C_i$ and
  the associated $1$-forms are written with boxes. The framework
  $G(p)$ has a standard lifting $L$ in the sense of
  Theorem~\ref{thm:maxwell} and it has a differential lifting denoted by $l_{\omega, p}$.
  }
  \label{fig:prism}
\end{figure}

The relation between polygonal liftings and differential liftings on 
self-stressed planar frameworks corresponding to planar graphs is given in
the following proposition.

\begin{proposition}
  Let $(G(p), \omega)$ be a self-stressed planar framework corresponding
  to a planar graph. Let furthermore $L$ denote a polyhedral lifting as in 
  Definition~\ref{def:liftings} and let $l_{\omega, p}$ be a differential
  lifting. Consider the ``normalized'' normal vector $\nu_C = (*, *, -1)$
  of the polyhedral lifting corresponding to the chamber $C$ and let
  $\alpha_C = l_{\omega, p}(C)$ be the associated $1$-form. Then
  \begin{equation*}
    \star \alpha_C 
    = 
    \pm \nu_C \cdot \Big(\begin{smallmatrix}dx\\dy\\0\end{smallmatrix}\Big),
  \end{equation*}
  where $\star$ denotes the \emph{Hodge star operator}, i.e.\ $\star(a\,
  dx + b\, dy) = -b\, dx + a\, dy$.
\end{proposition}
\begin{proof}
  This follows from Equations~\eqref{eq:normals} and~\eqref{eq:nei1}.
\end{proof}

The condition of the graph being planar in Theorem~\ref{thm:maxwell}
is not needed when we consider differential liftings instead of polyhedral
liftings.

\begin{theorem}
  \label{thm:bij}
  Let $G$ be any finite simple graph and $G(p)$ be a framework in $\RR^2$.
  The mapping $(G(p), \omega) \mapsto l_{\omega, p}$ is a bijection
  between the space of self-stresses and the space of differential
  liftings.
\end{theorem}
\begin{proof}
  First note that the space of self-stresses and the space of differential
  liftings are vector spaces.   
  The mapping is linear as well as surjective since $l_{\omega, p}$ is
  well defined by Theorem~\ref{thm:well}.
  
  Let two self-stresses have the same differential lifting. Then their
  difference is lifted to the ``zero lifting'', which, by
  Equation~\eqref{eq:nei1}, is possible only if the stress coefficients on all edges
  are zero.
\end{proof}

\section{Differential liftings in $\RR^n$}
\label{sec:n_dim}

Consider a framework $G(p)$ in $\RR^n$. For $n \geq 3$ the complement
$\RR^n \setminus G(p)$ does not decompose into disconnected chambers but
rather consists of one connected component. Consequently we replace the notion
of chambers by the notion of equivalence classes of oriented loops in the
complement of $G(p)$ to $\RR^n$. Here the base point of the loops can be
arbitrarily chosen to be any point of the complement.

Similarly to the differential liftings of Section~\ref{sec:2_dim} (which
are maps from $\pi_0(\RR^2 \setminus G(p))$ to $\Omega^1(\RR^2)$) we set
differential liftings in $\RR^n$ to be special maps of the fundamental
group $\pi_{n - 2}(\RR^n \setminus G(p))$ to $\Omega^1(\RR^n)$
(Definition~\ref{def:difflifting-n}). In our setting homotopic loops will
map to the same associated $1$-form. 

Given a vector $v = (v_1, \ldots, v_n)$, we identify it with the
differential $1$-form $dv = v_1\, dx_1 + \cdots + v_n\, dx_n$ via linear
combinations of coordinate differentials.

For a given framework, we consider maps from the fundamental group
$\pi_{n - 2}(\RR^n \setminus G(p))$ to the space of constant $1$-forms in
$\RR^n$. Once the map is fixed, we can say that the corresponding
$1$-forms are \emph{associated} with the homotopy classes of loops.

The number of how many times two disjoint manifolds are intertwined is
called \emph{linking number}. Let us recall its definition from the
original paper by Gauss~\cite[p.~605]{Gauss1877}.
In the literature, such as \cite[p.~53]{milnor1997topology}, one often
also finds another convention for the sign of this definition.

\begin{definition}
  \label{def:linking}
  Consider two disjoint oriented manifolds 
  $M_1, M_2 \subset \RR^{m_1 + m_2 + 1}$
  of respective dimensions $m_1$ and $m_2$. Then the \emph{linking map} is
  the map 
  \begin{equation*}
    \lambda : M_1 \times M_2 \longrightarrow S^{m_1 + m_2}
  \end{equation*}
  defined by $\lambda(x, y) = -(x - y)/\|x - y\|$. 
  If $M_1$ and $M_2$ are compact, oriented and without boundary, then the
  degree of $\lambda$ is called the \emph{linking number} $\lk(M_1, M_2)$.
\end{definition}

\begin{remark}
The orientation of the sphere $S^{m_1 + m_2}$ is important for the definition of the linking number. We fixed the orientation in such a way that 
at point
$(1,0,\ldots, 0) \in S^{m_1 + m_2}$ it coincides with the orientation of the basis $(e_2,\ldots,e_{m_1 + m_2+1})$.
In such settings the orientations of Hopf links in $\RR^3$
are as in Figure~\ref{hopf-Links}.
\end{remark}

\begin{figure}
    \centering
       \begin{overpic}[width=.5\textwidth]{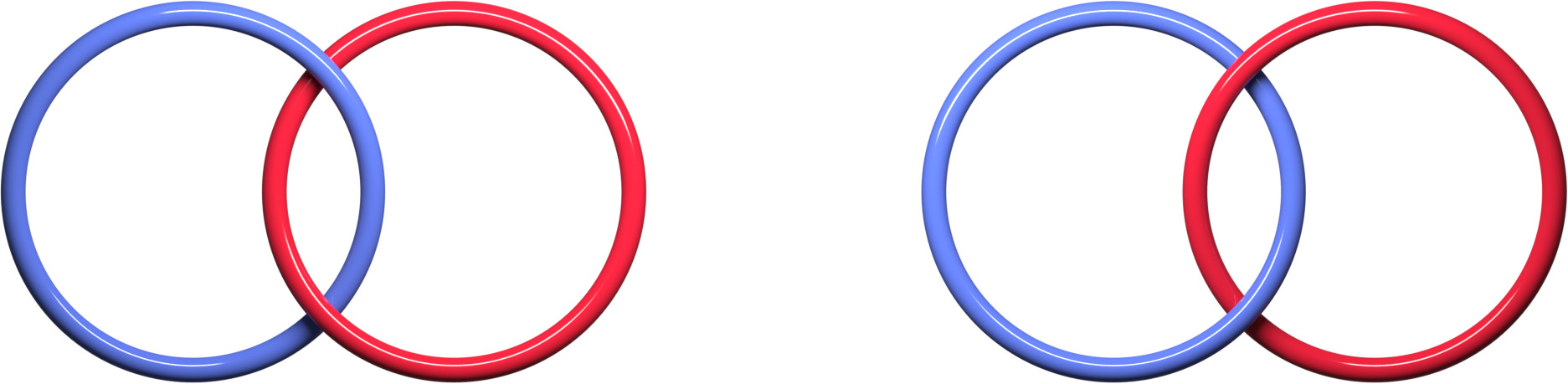}
         \put(-2,22){$M_1$}
         \put(38,22){$M_2$}
         \put(56,22){$M_1$}
         \put(96,22){$M_2$}
         \put(-3,7){%
           \begin{tikzpicture}
             \centerarc[myblue,line width=1pt,-latex](-0,0)(148:208:10mm);
           \end{tikzpicture}%
         }
         \put(42,7){%
           \begin{tikzpicture}
             \centerarc[myred,line width=1pt,-latex](0,0)(28:-28:10mm);
           \end{tikzpicture}%
         }
         \put(55,7){%
           \begin{tikzpicture}
             \centerarc[myblue,line width=1pt,-latex](-0,0)(148:208:10mm);
           \end{tikzpicture}%
         }
         \put(101,7){%
           \begin{tikzpicture}
             \centerarc[myred,line width=1pt,-latex](0,0)(28:-28:10mm);
           \end{tikzpicture}%
         }
       \end{overpic}
       \caption{A negatively (\emph{left}) and a positively (\emph{right})
       oriented Hopf link of two oriented circles in
       $\RR^2$.}\label{hopf-Links}
\end{figure}
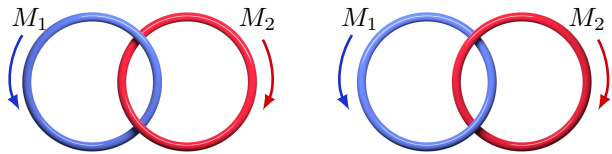

\begin{remark}
  Below, one of the manifolds will be a piecewise linear finite polyhedral
  $m$-sphere. In this case we consider the linking number of the limit of
  approximations by differentiable spheres in $\varepsilon$
  neighbourhoods of the piecewise linear sphere in the continuous topology
  while $\varepsilon$ tends to zero.
\end{remark}

The linking number is an integer which is homotopy
invariant and can be computed with the so called
multi-dimensional Gauss linking integral (see,
e.g.,~\cite{shonkwiler2011higher}). 
Coming back to our frameworks $G(p)$, for a simple cycle $\delta$ in
$G(p)$ and $\Gamma\in \pi_{n - 2}(\RR^n \setminus G(p))$ we set
\begin{equation*}
  \lk(\Gamma, \delta) = \lk(\gamma, \delta)
\end{equation*}
with $\gamma$ being any element of $\Gamma$. The linking number of two
manifolds is well defined since it is invariant under homotopies and hence
it does not depend on the choice of $\gamma \in \Gamma$.

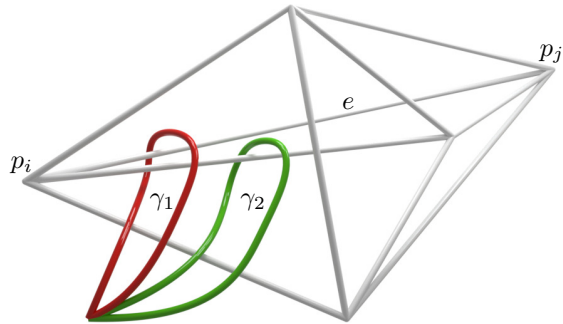
\begin{figure}[hb]
  \hfill
  \begin{overpic}[width=.47\textwidth]{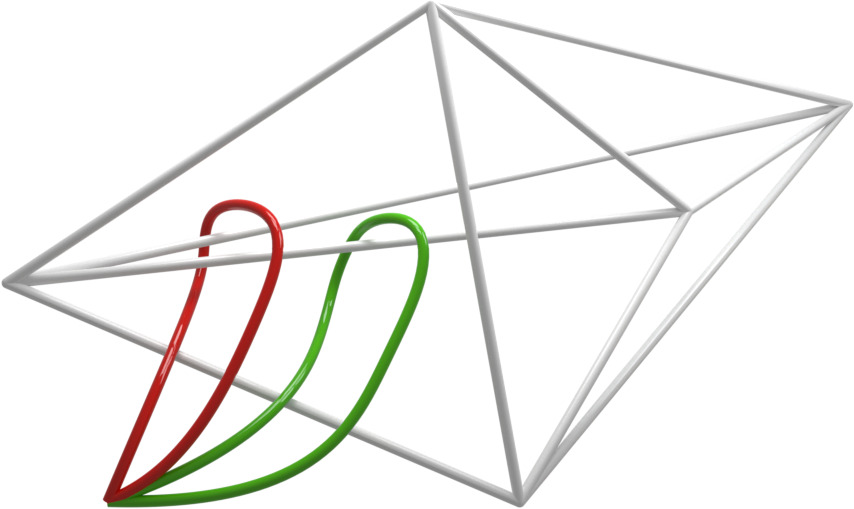}
    \put(-2,29){$p_i$}
    \put(97,50){$p_j$}
    \put(60,40){$e$}
    \put(24,22){$\gamma_1$}
    \put(41,22){$\gamma_2$}
  \end{overpic}
  \hfill{}
    \caption{Two loops $\gamma_1 \in \Gamma_1$ and $\gamma_2 \in \Gamma_2$
  from neighbouring homotopy classes $\Gamma_1$ and $\Gamma_2$ along the edge
  $e = p_i p_j$. There is a homotopy between $\gamma_1$ and $\gamma_2$
  which transversally crosses $e$ only once.}
  \label{fig:loops}
\end{figure}

We say that two homotopy classes $\Gamma_1, \Gamma_2 \in \pi_{n - 2}(\RR^n
\setminus G(p))$ are \emph{neighbours at an edge $e$} if there is a
homotopy between loops $\gamma_1 \in \Gamma_1$ and $\gamma_2 \in \Gamma_2$
in $\RR^n$ crossing the framework along $e$ transversally only once. See
Figure~\ref{fig:loops} for an illustration of loops from neighbouring
homotopy classes.
For each edge $e = p_i p_j$ there is a homotopy class which is a neighbour
to the neutral element of $\pi_{n - 2}(\RR^n \setminus G(p))$ along $e$
(see Figure~\ref{fig:fundamental-group}). We call it an \emph{elementary}
homotopy class. 

Consider a homotopy class $\Gamma \in \pi_{n - 2}(\RR^n \setminus G(p))$
and let $\Theta$ be an elementary homotopy class neighbouring the neutral 
element along $e = p_i p_j$. Then $\Gamma \pm \Theta$ is a neighbour of 
$\Gamma$ at $p_i p_j$.

Any homotopy class can be unwound by adding or subtracting elementary
loops to finally become the neutral element. Consequently, between any two
homotopy classes of $\pi_{n - 2}(\RR^n \setminus G(p))$ there is a
\emph{path} of neighbouring homotopy classes. Unwinding to the neutral
element and winding up to the target class would be such a path.

\begin{figure}[h]
  \hfill
  \begin{overpic}[width=.3\textwidth]{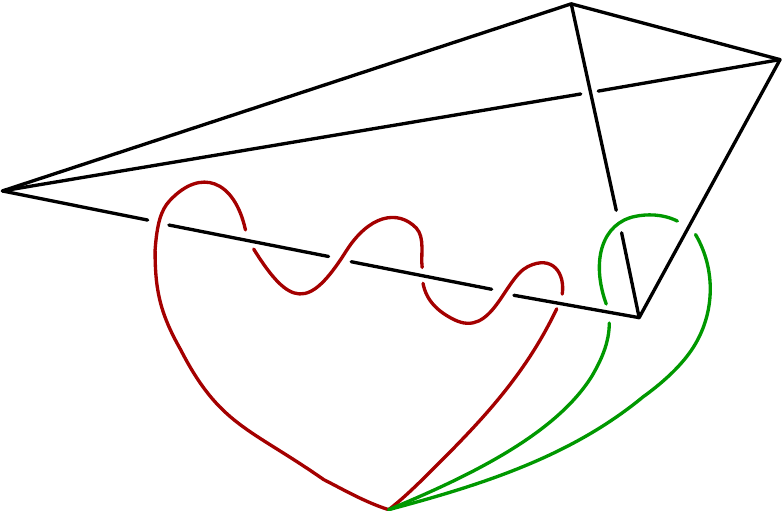}
    \put(25,19){$\gamma$}
    \put(84,11){$\theta$}
    \put(78,44){$e$}
  \end{overpic}
  \hfill
  \begin{overpic}[width=.3\textwidth]{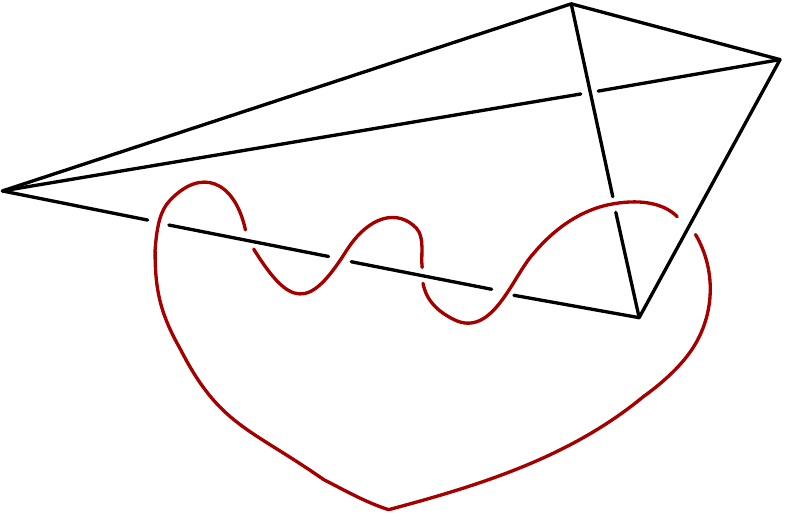}
    \put(25,19){$\gamma + \theta$}
    \put(78,44){$e$}
  \end{overpic}
  \hfill
  \begin{overpic}[width=.3\textwidth]{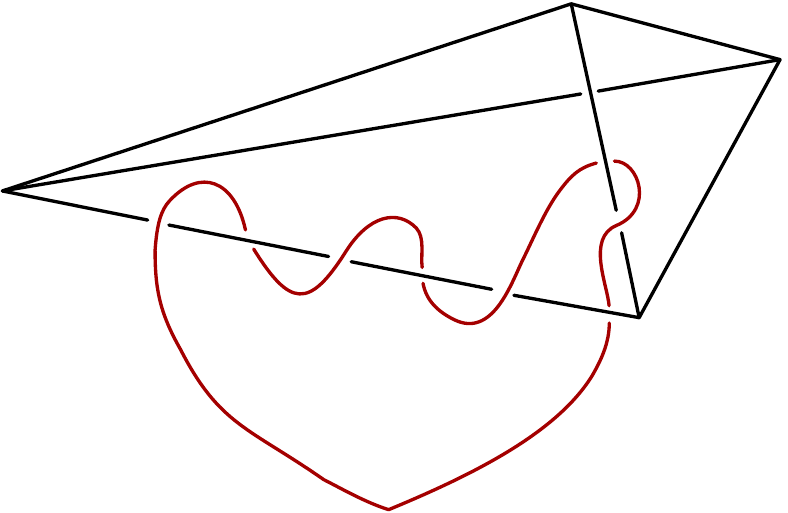}
    \put(25,19){$\gamma - \theta$}
    \put(77,46){$e$}
  \end{overpic}
  \hfill{}
  \caption{A framework $G(p)$ in $\RR^3$ and loops from $\pi_1(\RR^3
  \setminus G(P))$.
  \emph{Left:} The homotopy class $\Gamma$ is given by a representing loop
  $\gamma \in \Gamma$ and the elementary homotopy class $\Theta$ is given
  by an elementary loop $\theta \in \Theta$. The homotopy class $\theta$
  is a neighbour to the neutral element at edge $e$.
  \emph{Centre and Right:} The sums $\Gamma \pm \Theta$ are both neighbours
  to $\Gamma$ at edge $e$.
  }
  \label{fig:fundamental-group}
\end{figure}
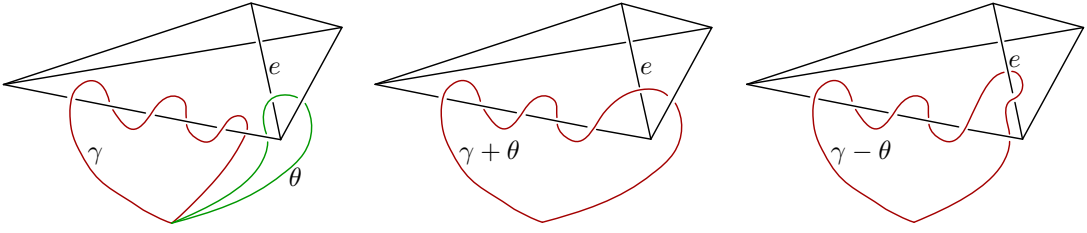

\begin{definition}
  \label{def:nei2}
  Consider a framework  $G(p)$ with a self-stress 
  $\omega$. Let $\Gamma_1$ and $\Gamma_2$ be two different homotopy classes
  of loops that are neighbours at an edge $p_i p_j$  with  stress
  coefficient $\omega_{ij}$. Fix a simple cycle $\gamma$ in $G(p)$ which
  includes the edge $p_i p_j$ with orientation from $p_i$ to $p_j$. 
  Let $\alpha_1, \alpha_2 \in \Omega^1(\RR^n)$ be two constant $1$-forms
  associated with $\Gamma_1$ and $\Gamma_2$. We say that $\alpha_1$ and
  $\alpha_2$ satisfy the \emph{neighbouring condition} if  
  \begin{equation}
    \label{eq:nei2}
    \alpha_2 - \alpha_1
    = 
    \big(\lk(\Gamma_1, \gamma) - \lk(\Gamma_2, \gamma)\big) \cdot
    \omega_{ij} \cdot d(p_i{-}p_j).
  \end{equation}
\end{definition}

This definition is neither dependent on the choice of cycle $\gamma$ nor
on the orientation of the edge $p_i p_j$: the choice of the simple cycle
in Definition~\ref{def:nei2} does not affect Equation~\eqref{eq:nei2}
since $\big({\lk}(\gamma, \Gamma_1) - \lk(\gamma,\Gamma_2)\big)$ stays the
same due to the orientation choice. Furthermore, the choice of the
orientation of the edge $p_ip_j$ does not affect the right hand side of
Equation~\eqref{eq:nei2} since the linking number changes sign and 
$p_j - p_i = -(p_i - p_j)$. 

\begin{remark}
  Note that since  $\Gamma_1$ and $\Gamma_2$ are neighbours, we have
  \begin{equation*}
  \lk(\Gamma_1, \gamma) - \lk(\Gamma_2, \gamma) = \pm 1.
  \end{equation*}
\end{remark}

\begin{definition}
  \label{def:difflifting-n}
  For the framework $G(p)$ and a self-stress $\omega$ of $G(p)$, we call a
  mapping 
  \begin{equation*}
    l_{\omega, p} : 
    \pi_{n - 2}(\RR^n \setminus G(p)) \longrightarrow \Omega^1(\RR^n)
  \end{equation*}  
  a \emph{differential lifting}, if $l_{\omega, p}(\Gamma)$ is a constant
  $1$-form for all $\Gamma \in \pi_{n - 2}(\RR^n \setminus G(p))$, if it
  satisfies the \emph{neighbouring condition} for each neighbouring homotopy
  classes, and if it takes zero on the neutral element $\Gamma_\infty \in
  \pi_{n - 2}(\RR^2 \setminus G(p))$.
\end{definition}

The following theorem is the higher dimensional analogue to
Theorem~\ref{thm:well} of the $2$-dimensional case.

\begin{theorem}
  \label{thm:well2}
  For every framework $G(p)$ with a self-stress $\omega$, there exists a
  unique differential lifting $l_{\omega, p}$.
\end{theorem}
\begin{proof}
For the simplicity of the exposition, we call the elements of $\pi_{n-2}(\RR^n\setminus G(p))$ chambers, similarly to the case $n=2$ considered in the previous section.

\noindent \textbf{Existence:}    Let us say that a \emph{path} between two chambers is a sequence of chambers that are neighbouring. The path is \emph{closed} if the starting chamber is the same as the end chamber.

The \emph{monodromy} of $\Omega^1(\RR^n)$ associated to a closed path
is the linear mapping defined by the summation of all forms on the right hand side of Equation~\eqref{eq:nei2}.
    
Let us show that all monodromies are zero.

\emph{Case 1}: Let us consider monodromies for the closed path $C_1,C_2,C_1$. They correspond to passing through an edge and going back, which will result in \begin{equation*}
     \alpha\mapsto \alpha+    
     \big(\lk(\Gamma_1,\gamma)-\lk(\Gamma_2,\gamma)\big)\omega_{ij} \cdot d(p_i{-}p_j)+ \big(\lk(\Gamma_2,\gamma)-\lk(\Gamma_1,\gamma)\big)\omega_{ij} \cdot d(p_i{-}p_j)=\alpha.
     \end{equation*}

\emph{Case 2}: Here we need to extend the definition of a closed path around a point $p_i$. Consider a small sphere $S^{n-1}$ centered at $p_i$ and let us fix the marked point $O_{n-1}$ for the homotopy group so that it lies in $S^{n-1}\setminus G(p)$. Let us consider a differential homotopy $H:[0,1]\times S^{n-2}\to S^{n-1}$ such that 
\begin{itemize}
    \item $H(\{0\}\times S^{n-2})=\{O_{n-1}\}$ and $H(\{1\}\times S^{n-2})=\{O_{n-1}\}$;
    \item $H([0,1]\times \{O_{n-2}\})=\{O_{n-1}\}$;
    \item $H:(0,1)\times S^{n-2}\setminus \{O_{n-2}\}\to S^{n-1}\setminus \{O_{n-1}\}$ is a differentiable, non-singular bijection with positive Jacobian;
    \item The pre-images of points  $G(p)\cap S^{n-1}$ have  pairwise different first coordinates.
\end{itemize}
This homotopy induces the  closed path between the chambers starting and ending in the identity chamber. We say that this path has a \emph{counterclockwise} direction. If we reverse the first coordinate, then we have \emph{clockwise} direction.

Let us consider a closed path around a point $p_i$ in clockwise direction, $C_1,\ldots, C_s,C_1$. Then the monodromy is
\begin{equation}
  \label{eq:differentials}
  \alpha\mapsto \alpha + \sum_{i=1}^s\big(\lk(\Gamma_i,\gamma)-\lk(\Gamma_{i+1},\gamma)\big)\omega_{ij} \cdot d(p_i{-}p_j)=
  \alpha + d\Big(\sum_{i=1}^s\omega_{i,j} (p_i{-}p_j)\Big)=
  \alpha.
\end{equation}
The first equality is true by the linearity of the differentials and since the differences between all linking numbers are  $+1$. The second equality holds by the equilibrium condition at $p_i$.

\emph{General case}:  The monodromy for any closed path is zero since it can be decomposed into the elementary closed paths considered in Cases 1 and 2.

Therefore, if we have two paths $\mathcal{C}_1$ and $\mathcal{C}_2$ with the same starting chamber $C_s$ and the same end chamber $C_e$, then the corresponding paths through neighbours will lead to the same differential form, since the closed path $\mathcal{C}_2$ followed by the reversed path $\mathcal{C}_1$ has zero monodromy.

\noindent \textbf{Uniqueness:} By definition, the form corresponding to the unbounded element  of $\pi_0(\RR^n\setminus G(p))$ (denoted by $C_\infty$) is zero. Then each path from $C_\infty$ to $C$ uniquely determines the form for $C$ by Equation~\eqref{eq:nei1}. 
\end{proof}

\begin{remark}
  Let $\Gamma = \Gamma_1 + \Gamma_2$. Then 
  \begin{equation*}
  l_{\omega, p}(\Gamma) = l_{\omega, p}(\Gamma_1) + l_{\omega, p}(\Gamma_2).
  \end{equation*} 
  Therefore, once we know the values of the differential liftings for a
  set of generators of $\pi_{n - 2}(\RR^n \setminus G(p))$, we get all the
  values for $l_{\omega, p}$ by the above additivity.
\end{remark}

\begin{theorem}
  \label{thm:bij2}
  Let $G$ be any graph and $G(p)$ be a framework in $\RR^n$.  The mapping
  $l_{\omega, p}$ is a bijection between the space of self-stresses and
  the space of differential liftings.
\end{theorem}
\begin{proof}
  The proof is the same as the proof of Theorem~\ref{thm:bij}.
\end{proof}

\begin{remark}
Note that geometric realisability of self-stresses was studied in~\cite{karpenkov+2021}.
\end{remark}

\begin{example}
Consider a framework with graph $K_5$ in 3-space with the following vertices:
$$
p_1=(0,0,0), \quad p_2=(1,0,0), \quad p_3=(0,1,0), \quad p_4=(0,0,1),
\quad p_5=(1,1,1).
$$
This framework has a one-dimensional space of self-stresses. Let $\omega$ be one of these self-stresses defined by the matrix:
$$
\omega=
\begin{bmatrix}
\mm 0 & \mm 2 & \mm 2 & \mm 2 &    -2\\
\mm 2 & \mm 0 &    -1 &    -1 & \mm 1\\
\mm 2 &    -1 & \mm 0 &    -1 & \mm 1\\
\mm 2 &    -1 &    -1 & \mm 0 & \mm 1\\
   -2 & \mm 1 & \mm 1 & \mm 1 & \mm 0\\
\end{bmatrix},
$$
where the stress coefficient $\omega_{i,j}$ is in position $(i,j)$.
Then the neighbouring condition along the edge $p_{i}p_{j}$ 
is 
$$
\alpha_2-\alpha_1=\big(\lk(\Gamma_1, \gamma) - \lk(\Gamma_2, \gamma)\big) \cdot\beta_{i,j}.
$$
The collection of $1$-forms $\beta_{i,j}$ is presented by the following matrix
(we consider $(x,y,z)$ as coordinates of 3-space):
$$
\beta=
\begin{bmatrix}
0 & -2dx & -2dy & -2dz &2dx+2dy+2dz\\
2dx & 0 &-dx+dy &-dx+dz & -dy-dz\\
2dy & dx-dy & 0 &-dy+dz & -dx-dz\\
2dz &dx-dz &dy-dz & 0 & -dx-dy\\
-2dx-2dy-2dz& dy+dz & dx+dz & dx+dy & 0\\
\end{bmatrix}.
$$
Consider the loops in Figure~\ref{K-5-loops}.
The homotopy class $\Gamma_1$ represented by $\gamma_1$ 
is a neighbour to the homotopy class of trivial loops at the edge $p_1p_2$, and hence
$$
l_{\omega, p}(\Gamma_1) = -2dx. 
$$
In order to get from the homotopy class of trivial loops to the 
homotopy class $\Gamma_2$ one may cross first the edge $p_2p_4$ 
and then the edge $p_2p_5$. So we have:
$$
l_{\omega, p}(\Gamma_2) = (-dx+dz)+(-dy-dz)=-dx-dy. 
$$
\begin{figure}
  \hfill
  \begin{overpic}[width=.3\textwidth]{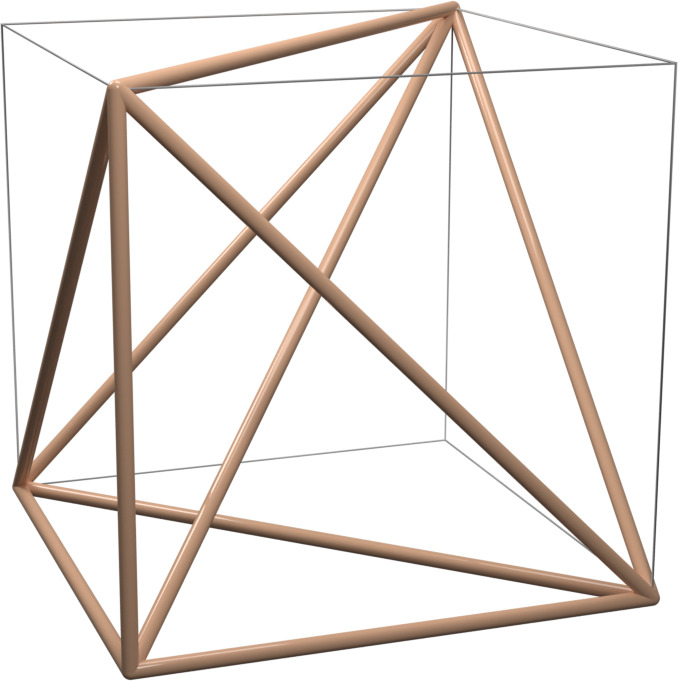}
    \put(10,0){$p_1$}
    \put(93,5){$p_2$}
    \put(-4,22){$p_3$}
    \put(10,90){\contour{white}{$p_4$}}
    \put(70,97){\contour{white}{$p_5$}}
  \end{overpic}
  \hfill
  \begin{overpic}[width=.3\textwidth]{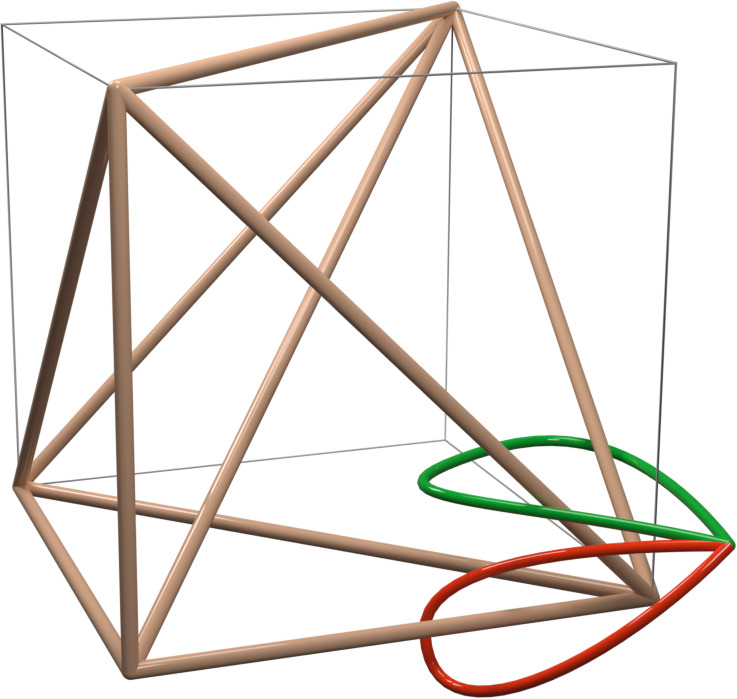}
    \put(10,0){$p_1$}
    \put(90,8){$p_2$}
    \put(-4,22){$p_3$}
    \put(10,88){\contour{white}{$p_4$}}
    \put(66,93){\contour{white}{$p_5$}}
    \put(51,0){$\gamma_1$}
    \put(93,30){$\gamma_2$}
  \end{overpic}
  \hfill{}
  \caption{Loops linked with $K_5$.}
  \label{K-5-loops}
\end{figure}
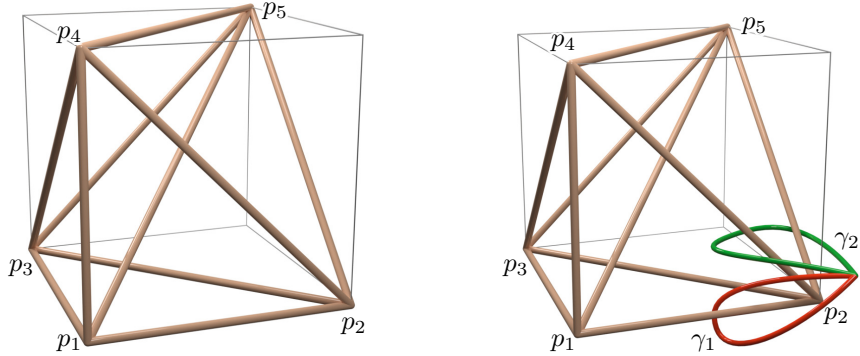

\end{example}

\begin{remark}
  Definition~\ref{def:nei2} coincides with Definition~\ref{def:nei1} in
  the case when $n = 2$. Here we consider $\pi_0$, which coincides with
  connected components.  
  In this case, $\Gamma_1$ and $\Gamma_2$ are represented by single points
  $x_1$ and $x_2$ in the corresponding chambers. Then we have
  \begin{equation*}
    \lk(x_1, \gamma) - \lk(x_2, \gamma)
    = 
    \wind(\gamma, x_1) - \wind(\gamma, x_2)
    =
    \sign(\det(\nn_{12}, p_i - p_j)),
  \end{equation*}
  where $\nn_{12}$ is the normal of the vector $p_i - p_j$ in the
  direction from the first to the second chamber and where `$\wind$' denotes
  the winding number.
  The first equation holds because the linking numbers coincide with the
  winding numbers. The second equation is a general property of the
  winding number.
\end{remark}

\section{Differential liftings of polytopal embeddings of complexes}
\label{sec:polytopal_embedding}

In this section we will generalise differential liftings to higher
dimensional frameworks. We will investigate structures which are
polyhedral CW complexes where all faces in which forces appear are of the
same dimension. We introduce $m$-frameworks and their equilibrium
stresses.

\subsection{The $m$-frameworks}
\label{subsec:k_frameworks}

Let us recall a general definition of $m$-frameworks
from~\cite{karpenkov+2022}. Our model for self-stressable $m$-frameworks
in this section will be based on the following structure. The term
\emph{plane} will refer to an affine subspace in $\RR^n$. In the
following let $n > m \geq 1$.

\begin{definition}
  %\label{def:dframeworkdefinition}
  An \emph{$m$-framework} $\mathcal{F} = (E, F, I, \mathbf{n})$ consists of
  $E$, a collection of $(m{-}1)$-dimensional planes in $\RR^n$; $F$, a
  collection of $m$-dimensional planes in $\RR^n$; a subset $I \subset
  \{(e, f) \in E \times F \mid e \subset f\}$; a function $\mathbf{n}$
  assigning to each pair $(e, f)\in I$, a unit vector $\mathbf{n}(e, f)$
  which is contained in $f$ and which is orthogonal to $e$. We call planes
  from $F$ \emph{faces} and planes from $E$ \emph{edges}.\end{definition} 
 
 \begin{definition}
 Let $\mathcal{F} = (E, F, I, \mathbf{n})$ be an $m$-framework. A 
\emph{force-load} $\omega$ on $\mathcal{F}$ is any function
$\omega : F \to \RR$. A framework $\mathcal{F}$ together with a force-load
$\omega$ is said to be in \emph{equilibrium} if for every $e \in E$ we have
\begin{equation*}
  %\label{self-stress-d-framework}
  \sum_{(e, f) \in I} \omega(f)\mathbf{n}(e, f) = 0.
\end{equation*}
Such a force-load is called an \emph{equilibrium force-load} for
$\mathcal{F}$.
\end{definition}

Note that there is a difference between force-loads and stresses for
graphs. The latter is related the length of the edges whereas force-loads
are not. For now stresses are not defined as faces of $m$-frameworks are
planes and do not have specified shapes or magnitudes. After introducing
polytopal complexes below we could consider stresses as well. 

\begin{definition}
  %\label{tensegritydefinition}
  An $m$-framework is said to be \emph{self-stressable} or a
  \emph{tensegrity} if there exists a non-zero
  equilibrium force-load on it.
\end{definition}

\subsection{Polytopal embeddings of complexes}

An \emph{$m$-dimensional polytope} is the convex hull of a finite set of
points spanning $\RR^m$. The boundary is piecewise linear so we have a
natural structure of faces. The faces of dimension $m{-}1$ are called
\emph{facets}.

\begin{definition}
Consider a finite collection $F$ of $m$-dimensional polytopes
in $\RR^n$
whose intersection is either a facet or a $q$-dimensional
polytope with $q \leq m - 2$. 
Let $E$ be the collection of all facets of the polytopes in $F$ (here we identify the same facets of distinct polytopes). 
We also require that the planes of any two facets in $E$ are distinct.
Then the pair $\mathcal{C} = (F, E)$ is called   an \emph{$m$-dimensional polytopal complex} in
$\RR^n$. \end{definition}

\begin{definition}
  The $m$-framework $\mathcal{F}_C = (E_C, F_C, I_C, \mathbf{n})$
  \emph{associated} with an $m$-dimensional polytopal complex $C$ is
  defined by 
  \begin{itemize}
    \item $E_C$, the collection of $(m{-}1)$-dimensional planes spanned
      by facets of $E$;
    \item $F_C$, the collection of $m$-dimensional planes spanned by
      polytopes of $F$;
    \item $I_C = \{(e, f) \in E_C \times F_C \mid \text{the facet}\ e\
      \text{is a facet of polytope}\ f\}$; 
    \item a function $\mathbf{n}$ assigning to each pair $(e, f) \in I_C$
      a unit vector $\mathbf{n}(e, f)$ which is contained in $f$, is
      orthogonal to $e$, and points to the interior of the polytope.    
  \end{itemize}
\end{definition}

\begin{example}
  Let us consider a very broad family of examples arising from parallel
  meshes in $\RR^3$. We say that two polyhedral surfaces $P_1$ and
  $P_2$ are \emph{parallel} if $P_1$ and $P_2$ have the same combinatorics 
  and such that corresponding edgwes are parallel (consequently also
  corresponding faces are parallel). Let us connect the corresponding
  parallel edges of such a pair of meshes.
  Then we get trivalent $2$-dimensional polytopal complexes (see an
  example in Figure~\ref{fig-parallel}). For further details on parallel
  surfaces and related discrete differential geometry we refer
  to~\cite{bobenko-pottmann-wallner} and~\cite{karpenkovWallner-2014}.
\end{example}

\begin{figure}
  \begin{overpic}[width=.5\textwidth]{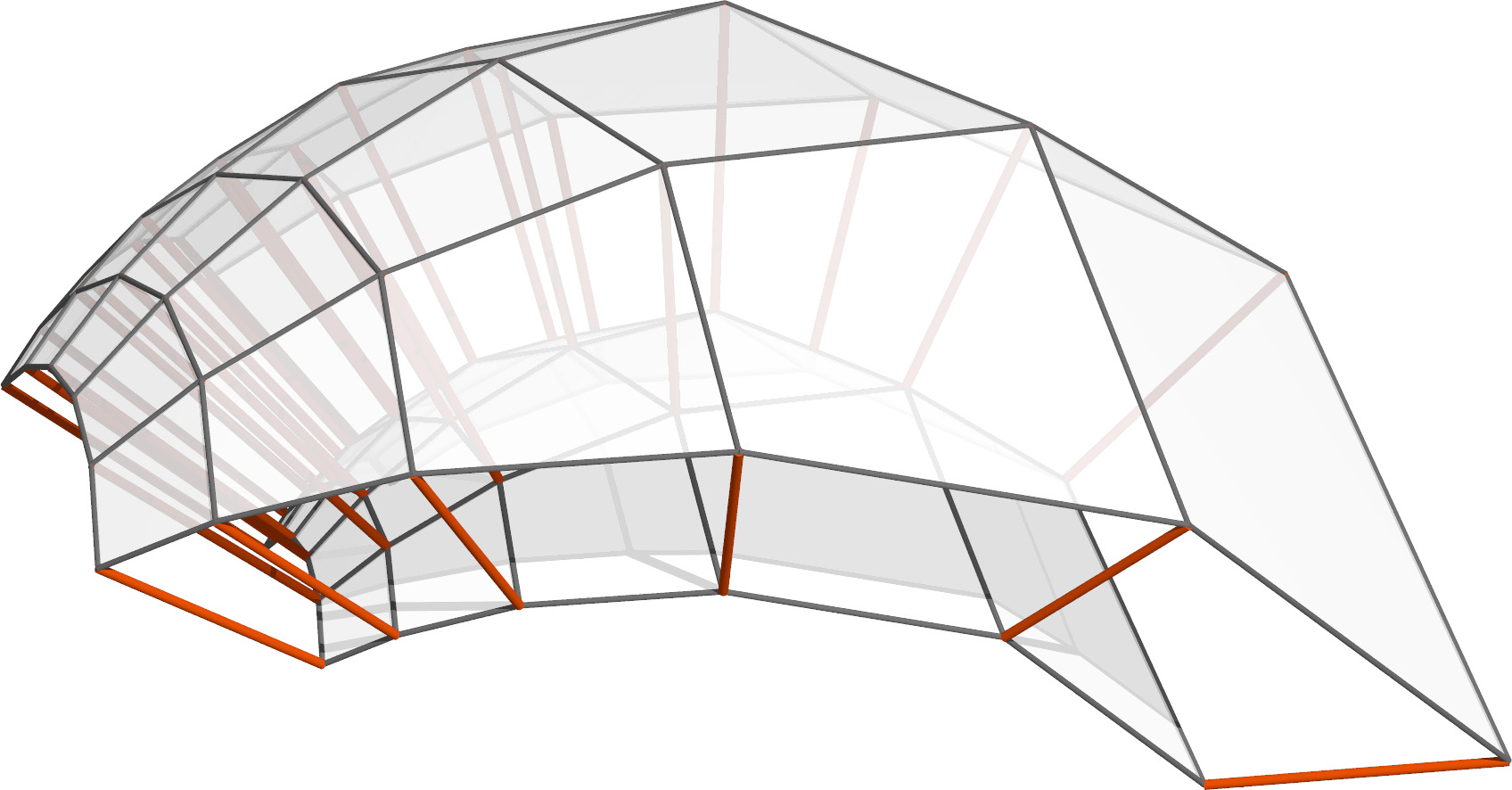}
  \end{overpic}
\caption{An example of a pair of parallel meshes.}\label{fig-parallel}
\end{figure}

\subsection{Differential liftings of $m$-frameworks}

Consider an $m$-dimensional polytopal complex $C$.  While for graphs we
worked with $1$-forms, we now work with $m$-forms. As before, we will
study fundamental groups $\pi_{n - m - 1}(\RR^n \setminus C)$.
Again we can choose the base point to be any point of the complement. Note
that homotopic $(n{-}m{-}1)$-spheres will be associated with the same $m$-form. 

For an oriented $m$-dimensional polytope $f$ we set the differential form
\begin{equation*}
  \alpha_f = de_1 \wedge \ldots \wedge de_m,
\end{equation*} 
where $e_1, \ldots, e_m$ is an orthonormal basis of the span of $f$ with
positive orientation.

For a given $m$-dimensional polytopal complex $C$, we consider maps from
the fundamental group $\pi_{n - m - 1}(\RR^n \setminus C)$ to the space of
constant $m$-forms in $\RR^n$. Once the map is fixed, we say that the
corresponding $m$-forms are \emph{associated} with the homotopy classes of
spheres.

We say that two homotopy classes of spheres in $\RR^n \setminus C$ are
\emph{neighbours at a facet} $e$ if there is a homotopy of spheres in $\RR^n$
crossing the framework along $e$ only once (transversally).

An \emph{$m$-cycle} in $C$ is an oriented sub-complex of $C$ homeomorphic to the unit
$m$-sphere.

\begin{definition}
  \label{def:nei3}
  Consider an $m$-dimensional polytopal complex $C$ with a non-zero equilibrium force-load
  $\omega$ on the associated $m$-framework $\mathcal{F}_C$. Let $\Gamma_1$
  and $\Gamma_2$ be two different homotopy classes of spheres that are
  neighbours at an oriented facet $f$ with stress coefficient
  $\omega_{f}$. Fix a simple $m$-cycle $\gamma$ in $C$ which includes the
  facet $f$ with orientation of $f$. 
  Let $\alpha_1$ and $\alpha_2$ be two constant $m$-forms associated with
  $\Gamma_1$ and $\Gamma_2$. We say that $\alpha_1$ and $\alpha_2$ satisfy
  the \emph{neighbouring condition} if  
  \begin{equation}
    \label{eq:nei3}
    \alpha_2 - \alpha_1
    = 
    (\lk(\Gamma_1, \gamma) - \lk(\Gamma_2, \gamma)) \cdot \omega_{f} 
    \cdot \alpha_f.
  \end{equation}
\end{definition}

\begin{remark}
  Note that since $\Gamma_1$ and $\Gamma_2$ are neighbours, we have
  \begin{equation*}
  \lk(\Gamma_1, \gamma) - \lk(\Gamma_2, \gamma) = \pm 1.
  \end{equation*}
\end{remark}

\begin{remark}
   The choice of the simple $m$-cycle in Definition~\ref{def:nei3} does
   not affect Equation~\eqref{eq:nei3} since
   $(\lk(\gamma, \Gamma_1) - \lk(\gamma, \Gamma_2))$ does not change due to the orientation choice.

   Note also that the choice of the orientation of the facet $f$ does not
   affect the right hand side of Equation~\eqref{eq:nei3} since the
   linking number and the form $\alpha_f$ simultaneously change their signs. 
\end{remark}

\begin{definition}
  \label{def:diflfting-polyhedra}
  Consider an $m$-dimensional polytopal complex $C$ with a non-zero force-load
  $\omega$ on the associated $m$-framework $\mathcal{F}_C$. The
  \emph{differential lifting} $l_{\omega, C}$ is the $m$-form which 
  \begin{itemize}
  \item
  is
  constant on the connected components of the complement of $C$ to $\RR^n$;
  \item satisfies the neighbouring condition for all neighbouring homotopy classes of spheres;
  \item 
  is zero on the homotopy class of spheres that includes the constant map.
\end{itemize}

  Denote by $l_{\omega, C}(\Gamma)$ the form corresponding to $\Gamma \in
  \pi_{n - m - 1}(\RR^n\setminus C)$.
\end{definition}

\begin{theorem}
  \label{thm:well3}
  For every $m$-framework $\mathcal{F}_C$ with a force-load $\omega$,
  there exists a unique differential lifting of $C$.
\end{theorem}
\begin{proof}
  For the simplicity of the exposition, we call the elements of $\pi_{n-m-1}(\RR^n\setminus G(p))$ the \emph{chambers}, similarly to the case $m=1$ considered in the previous section.

\noindent \textbf{Existence:}    Let us say that a \emph{path} between two
  chambers is a sequence of chambers that are neighbouring. The path is \emph{closed} if the starting chamber is the same as the end chamber.

    The \emph{monodromy} of $\Omega^m(\RR^n)$ associated to a closed path
is the linear mapping defined by the summation of all forms that are second summands on the right hand side of Equation~\eqref{eq:nei3}.

Let us show that all monodromies are zero.

\emph{Case 1}: Let us consider monodromies for the closed path $\Gamma_1,\Gamma_2,\Gamma_1$. They correspond to passing through a polytope $f$ and going back, which will result in \begin{equation*}
     \alpha\mapsto \alpha+    
     \big(\lk(\Gamma_1,\gamma)-\lk(\Gamma_2,\gamma)\big)\omega_{ij} \cdot \alpha_f+ \big(\lk(\Gamma_2,\gamma)-\lk(\Gamma_1,\gamma)\big)\omega_{ij} \cdot \alpha_f=\alpha.
     \end{equation*}

\emph{Case 2}: Here we need to extend the definition of a closed path
  around a facet $e$ in clockwise direction. It is done similarly to the
  case $m=1$ by restricting our $(n{-}m{-}1)$-spheres to the orthogonal
  complement to $e$, which is $(n{-}m{+}1)$-dimensional.

Let us consider a closed path around a facet $e$ in clockwise direction, $\Gamma_1,\ldots, \Gamma_s,\Gamma_1$. Then the monodromy is
\begin{equation*}\alpha\mapsto \alpha + \sum_{i=1}^s\big(\lk(\Gamma_i,\gamma)-\lk(\Gamma_{i+1},\gamma)\big)\omega_{ij} \cdot\alpha_e\wedge \nn(e,f)=
 \alpha + \alpha_e\wedge d\Big(\sum_{i=1}^s\omega_{i,j}  \nn(e,f)\Big)=
\alpha.\end{equation*}
The first equality is true by the linearity of the differentials and since the differences between all linking numbers are  $+1$. The second equality holds by the equilibrium condition at $e$.

\emph{General case}:  The monodromy for any closed path is zero since it is decomposable into the elementary closed paths considered in Cases 1 and 2.

Therefore, if we have two paths $\mathcal{C}_1$ and $\mathcal{C}_2$ with
  the same starting chamber $\Gamma_s$ and the same end chamber $\Gamma_e$, then the
  corresponding paths through neighbours will lead to the same differential form, since the closed path $\mathcal{C}_2$ followed by the reversed path $\mathcal{C}_1$ has zero monodromy.

\noindent \textbf{Uniqueness:} This follows from exactly the same argument as in the proof of Theorem~\ref{thm:well3}. 
 \end{proof}

\begin{remark}
    Consider $\Gamma,\Gamma_1,\Gamma_2 \in \pi_{n - m - 1}(\RR^n \setminus C)$ satisfying the identity $\Gamma = \Gamma_1 + \Gamma_2$. Then 
    \begin{equation*}
    l_{\omega, C}(\Gamma) 
    =
    l_{\omega, C}(\Gamma_1) + l_{\omega, C}(\Gamma_2).
    \end{equation*}
    Therefore, once we know the values of the differential liftings for a
    set of generators of $\pi_{n - m - 1}(\RR^n \setminus C)$, we get all
    the values for $l_{\omega, C}$ by the above additivity.
\end{remark}

\begin{theorem}
  Let $C$ be an $m$-dimensional polytopal complex.
  The mapping $(C, \omega) \to l_{\omega, C}$ is a bijection between the
  space of force-loads and the space of differential liftings.
\end{theorem}

\subsection{Stresses and force-loads, and relation to the case $m = 1$}
\label{subsec:stress}

The general situation we considered in this section has one minor
difference from the settings of the previous two sections. Here, we
considered force-loads rather than stresses. As we know, the equilibrium
conditions for stresses are
\begin{equation*}
\sum\limits_{\{j \mid j \ne i\}} \omega_{ij} (p_i - p_j) = 0,
\end{equation*}
while the equilibrium conditions for force-loads are normalised as 
\begin{equation*}
\sum\limits_{\{j \mid j \ne i\}} 
\omega_{ij} \frac{p_i - p_j}{\|p_i - p_j\|} = 0.
\end{equation*}
For polytopes, we could also consider the corresponding stresses defined
by the equilibrium condition
\begin{equation*}
  %\label{eq:ssd}
  \sum_{(e, f) \in I} \textrm{Vol}(f_*) \omega(f) \mathbf{n}(e, f) = 0,
\end{equation*}
where $f_*$ is the facet corresponding to the $m$-plane $f$. 

Now let us recall the Equation~\eqref{eq:nei2}:
\begin{equation*}
  \alpha_2
  =
  \alpha_1 + (\lk(\Gamma_1, \gamma) - \lk(\Gamma_2, \gamma)) 
  \cdot \omega_{ij} \cdot d(p_i{-}p_j).
\end{equation*}
This equation coincides with the non-normalised version of
Equation~\eqref{eq:nei3}:
\begin{equation*}
  \alpha_2
  =
  \alpha_1 + \textrm{Vol}(f_*) \big((\lk(\Gamma_1, \gamma) - 
  \lk(\Gamma_2, \gamma)) \cdot \omega_{f} \cdot \alpha_f\big).
\end{equation*}
Here, we have added the extra factor $\textrm{Vol}(f_*)$ to the second
summand of the right hand side of Equation~\eqref{eq:nei3}.

\section{On liftings in general settings}\label{Definition of liftings in general settings}

In this section, we introduce general liftings.
In Subsection~\ref{Liftings on regular spheres} we study liftings on regular spheres defined by integration of the differential liftings. 
Further in Subsection~\ref{Grassmannian liftings}, 
we discuss a slightly different notion of liftings on the Grassmannians of planes.

For simplicity we deal with polytopal embeddings of complexes that also cover all  1-frameworks without self-intersections. The theory for self-intersecting frameworks is very similar, so we omit it here. The problem for general complexes remains open.

\subsection{Liftings on regular spheres}\label{Liftings on regular spheres}

In this subsection, we define liftings in the class of regular $k$-spheres in $\mathbb R^n$ 
(i.e.\ differential embeddings of the unit 
$k$-sphere with maximal ranks of Jacobi matrices at every point; note that we allow self-intersections).
In this section we do not consider the case $k=0$.

\subsubsection{Definition of liftings}
Denote the space of regular $k$-spheres passing through $p$
by $\Omega_k(p)$.

\begin{definition}
Consider a regular sphere $\gamma$ in $\Omega_k(p)$.
The homotopy corresponding to the linear homothety of $\gamma$ (at time $1$) to $p$ (at time $0$)
is called a {\it cone homotopy}. We denote it by 
$\mathcal H_p(\gamma)$.
\end{definition}

\begin{definition}
Let $\alpha$ be a decomposable $m$-form defined on $ S\subset \mathbb R^n$ and a point $r\in S$.
Consider some oriented orthonormal basis
$(e_1,\ldots, e_n)$ in which we have 
\begin{equation*}
\alpha|_r=f\cdot e_1\wedge\ldots\wedge e_m.
\end{equation*}
We say that the form
\begin{equation*}
\beta|_r=f\cdot e_{m+1}\wedge\ldots\wedge e_n
\end{equation*}
is the {\it Hodge star} of $\alpha$ at $p$.

\noindent{An arbitrary $m$-form $\alpha$ can be represented as a sum of decomposable forms. Here we define the {\it Hodge star} for $\alpha$ by additivity.}

\noindent
{We say that the form $\beta$ is the {\it Hodge star}
of $\alpha$ if it is the Hodge star at every point of $S$.
Denote the Hodge star by $\star \alpha$.}
\end{definition}

\begin{definition}
Consider an $m$-dimensional polytopal complex $C$ in $\mathbb R^n$ with a force-load $\omega$ on the associated framework $\mathcal F_C$. Let $n>m+1$ and let  $p$ be in the complement of $\mathcal F_C$.
 Consider the function $L_{\mathcal F_C,\omega,p}$ defined on the set of all 
regular spheres of dimension $n-m-1$ as follows:
\begin{equation*}
L_{\mathcal F_C,\omega,p}(\gamma)=
\int\limits_{\mathcal H_p(\gamma)}
\star\alpha[C,\omega].
\end{equation*}
We say that the function $L_{\mathcal F_C,\omega,p}$ is the {\it lifting} of $(\mathcal F_C,\omega)$.
\end{definition}

\begin{remark}
Liftings are continuous in the continuous topology on $\Omega_{n-m-1}(p)$.
There is no significant dependence on the choice of the point $p$ (due to Proposition~\ref{homotopy-proposition-1} below).
\end{remark}

Let us collect several simple statements on integration of differential forms over
regular embeddings without boundary.

\subsubsection{Co-dimension 0 obstacles cycle integration}

Let us start with the following proposition.

\begin{proposition}\label{homotopy-proposition-1}
Let $H_1$ and $H_2$ be two regular homotopies from $\gamma_1$ to $\gamma_2$ for the regular 
$(n-m-1)$-spheres, and let $\alpha$ be any constant $(n-m)$-form
in $\mathbb R^n$.
Then
\begin{equation*}
\int_{H_1}
\alpha
=\int_{H_2}
\alpha.
\end{equation*}
\end{proposition}

\begin{proof}
Let $H\times I$ be a regular homotopy between two regular homotopies~$H_1$ and~$H_2$.
Then by Stokes' formula (see, e.g. in~\cite{Arnold1989}) we have
\begin{equation*}
\int_{H_1} \alpha -\int_{H_2} \alpha
=
\int_{H\times I}d\alpha
=
\int_{H\times I}0=0.
\end{equation*}
This implies the statement of the proposition.
\end{proof}

\subsubsection{Co-dimension 1 obstacles cycle integration}

Let us continue with the following definition.

\begin{definition}
Let $\ell$ be an $m$-dimensional plane in $\mathbb R^n$,
we say that a {\it Heaviside form} on the space of regular $m$-spheres with respect to $\ell$ is defined by the following formula
$$
\chi_\ell(\gamma)
    = 
    \lk(\ell, \gamma) \cdot \star d{\ell},
$$
where $d\ell$ is the volume $m$-form on $\ell$.
\end{definition}

Note that $\chi_\ell$ is an $(n{-}m)$-form.

\begin{proposition}\label{homotopy-proposition-2}
 Let $H_1$ and $H_2$ be two regular homotopies from $\gamma_1$ to $\gamma_2$ for the regular 
$(n{-}m{-}1)$-spheres in $\mathbb R^n$,
and let $\ell$ be an $m$-dimensional plane in $\mathbb R^n$.
Then
\begin{equation*}
\int_{H_1}
\chi_\ell
=\int_{H_2}
\chi_\ell.
\end{equation*}   
\end{proposition}

\begin{proof}
{\it Case 1}. Assume that the homotopies $H_1$ and $H_2$ from $\gamma_1$ to $\gamma_2$ cross the plane $\ell$ at single points $p_1$ and $p_2$ (at parameter values $t_1$ and $t_2$) respectively.

Denote the closed loop formed by $H_1$ and $H_2$ 
(where $H_2$ is passed in the opposite direction).
The parameter values $t_1$ and $t_2$ (for $H_1$ and $H_2$) divide this loop into two parts
$S_1$ and $S_2$ containing $\gamma_1$ and $\gamma_2$, respectively.

Consider the homotopy which is a parallel transform
of $H_1(*,t_1)$ by the vector $p_2{-}p_1$
(here we add an infinitesimal tubular neighbourhood 
from the base point for the $(n{-}m{-}1)$-spheres along the vector $p_2-p_1$, which will not affect the value of the integrals).
Denote the corresponding path by $S_3$ and its inverse by $S_4$.

Finally, let us pick some regular homotopy
that joins $H_1(*,t_1)$ transported to $p_2$
and $H_2(*,t_2)$. Denote this path by  $S_5$ and its inverse by $S_6$ (we assume that this homotopy  keeps the point $p_2$ fixed and the other points of
the homotopy do not intersect $\ell$), see Figure~\ref{fig-ok-2-3}. 

\begin{figure}
\begin{center}
\begin{tabular}{ccc}
$
\begin{array}{c}
\includegraphics[scale = .8]{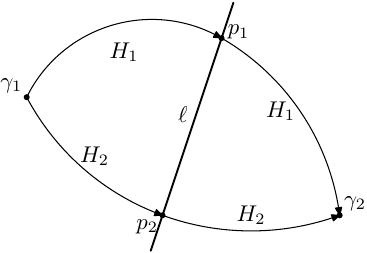}
\end{array}
$
&
$\qquad \longrightarrow \qquad $
&
$
\begin{array}{c}\includegraphics[scale = .8]{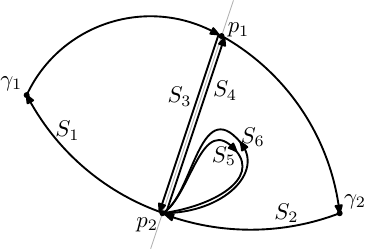}
\end{array}$
\\
\end{tabular}
\end{center}
\caption{Homotopies $H_1$ and $H_2$ and the corresponding paths $S_1,\ldots, S_6$.}\label{fig-ok-2-3}
\end{figure}

Denote
$$
\alpha=\chi_\ell(\gamma_1) \qquad \hbox{and} \qquad
\beta=\chi_\ell(\gamma_2)-\chi_\ell(\gamma_1).
$$
Then we have:
\begin{equation*}
\begin{aligned}
\int_{H_1} \chi_\ell
-\int_{H_2}\chi_\ell
&= \int_{S_1}\alpha+ \int_{S_2}(\alpha+\beta) \\
&= \int_{S_1\cup S_3 \cup S_5} \alpha + \int_{S_2\cup S_6 \cup S_4}(\alpha+\beta) 
 -\int_{S_3}\alpha \\&-\int_{S_5}\alpha - \int_{S_4}(\alpha+\beta) - \int_{S_6}(\alpha+\beta) \\
&= -\int_{S_3}\alpha - \int_{S_4}(\alpha+\beta) \\
&= -\int_{S_3\cup S_4}\alpha - \int_{S_4}\beta = -\int_{S_4}\beta = 0.
\end{aligned}
\end{equation*}

The second and the third equalities hold by Proposition~\ref{homotopy-proposition-1}.
The last equality holds since we integrate along the kernel of the form (recall that $p_1p_2$ is orthogonal to the integral levels of the form).

{\noindent {\it Case 2.}}
Assume that the homotopy 
$H_1$ does not cross the line $\ell$
while $H_2$ crosses $\ell$  transversally at two different points $p_1$ and $p_2$, respectively.

This case is reduced to the previous one
by considering it as a single homotopy $H_1$ and then $H_2$ based in the opposite direction.
The further proof literally repeats the previous case.

{\noindent {\it General Case}}
is a composition of Cases~1 and~2.
\end{proof}

\subsubsection{Co-dimension 2 obstacles cycle integration}

Next we show the following monodromy statement for 
$(m-1)$-dimensional facets in 
$m$-dimensional polytopal complexes.

\begin{proposition}\label{homotopy-proposition-3}
Let us consider a homotopy from a point to itself by $m$-spheres along the $(m+1)$-sphere.
We assume that the convex interior of the 
$(m+1)$-sphere contains the only point of facets at the center of the sphere.
Then the integral of $\star\alpha[C,\omega]$ over this sphere is zero.
\end{proposition}

\begin{proof}
Let us consider the cone homotopy that shrinks the $(m+1)$-sphere to its center along the cone. 
On the one hand by Propositions~\ref{homotopy-proposition-1}
and~\ref{homotopy-proposition-2},
this will not change the value of the integral.
On the other hand we are shrinking to a single point, and hence the limiting value of the integral tends to 0.

Therefore, the integral along the $(m+1)$-sphere around a facet $e$  is zero.
\end{proof}

\subsubsection{Continuity of liftings}
Let us finally formulate the following general property of liftings.

\begin{theorem}
The lifting is a continuous function in the continuous topology on
$\Omega_{n-m-1}(p)$.
\end{theorem}

\begin{proof}
Let us outline the proof.
Consider a small regular perturbation
(within any small distance in the space of continuous homotopies)
changing the original regular 
homotopy to the one avoiding all $(m-1)$-dimensional facets of the complex and 
transversally intersecting the complex in finitely many points.
The resulting homotopy can be decomposed into the homotopies studied in Propositions~\ref{homotopy-proposition-1},
\ref{homotopy-proposition-2}, and~\ref{homotopy-proposition-3}.
\end{proof}

\subsection{Grassmannian liftings}\label{Grassmannian liftings}
We conclude this section with the following remarkable representation of liftings as functions on affine Grassmannians.

The informal idea behind the below definitions is as follows.
First of all, let us restrict the sets of regular $m$-spheres to geometric $m$-spheres passing through a given point $p$.
Secondly, let us ``move'' $p$ to infinity.
That will bring us to the affine Grassmannian of oriented $m$-planes $\Gr (m,n)$.

First of all we give a general definition of distances between planes.

\begin{definition}
The {\it distance between the planes} $\pi_1$ and $\pi_2$ 
spanned by $(p_1,\ldots, p_n)$ and $(q_1,\ldots,q_m)$ respectively is the following number
\begin{equation*}
\dist(\pi_1,\pi_2):=\frac{\det(q_2-q_1,q_3-q_1,\ldots,q_m-q_1,p_1-q_1,\ldots,p_n-q_1)}
{\det(q_2-q_1,q_3-q_1,\ldots, q_m-q_1)\cdot \det(p_2-p_1,p_3-p_2,\ldots, p_n-p_1)}.
\end{equation*}
%We denote it by $\dist(\pi_1,\pi_2)$.
\end{definition}

Informally, distance functions serve as linear functions over Grassmannians. 

In order to define the values of the liftings at all lines of the Grassmannian we use the following notion. 
We consider the Grassmannian as a subset of the projective plane $\mathbb RP^N$ defined by Pl\"ucker coordinates of planes. 
So we have a natural structure of the differential manifold on Grassmannians.

\begin{definition}
Let $\mathcal F_C$ be a framework of an $m$-dimensional polytopal complex $C$ in $\mathbb R^n$. 
A differentiable path in $\Gr(n{-}m{-}1, n)$ is {\it $\mathcal F_C$-simple} if it intersects 
$\mathcal F_C$ only finitely many times;
each of the intersections, except the very last point of the path, is in the interior of some %of the 
$m$-dimensional face and it is transversal to this face (see Figure~\ref{figure-ok-1}). The last line in the path can be taken arbitrarily.
\end{definition}

\begin{figure}
\includegraphics[scale = 1.1]{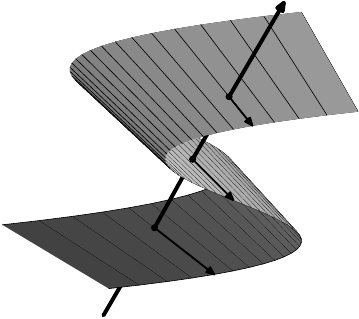}
\caption{Transversal intersections of a differentiable path with one of the elements of $\mathcal F_C$}\label{figure-ok-1}
\end{figure}

Let us define a lifting along $\mathcal F_C$-simple paths.

\begin{definition}
Consider an $\mathcal F_C$-simple path $H$ crossing an $m$-dimensional face $f$ of $\mathcal F_C$. 
We say that the {\it sign of the crossing} is the number
$$
\mu(H,f):=\sign\Big(\frac{\partial(\dist(H(t),f))}{\partial t}\Big).
$$
We say that a simple path $H$ intersects an $m$-dimensional face $f$ of $\mathcal F_C$ with {\it sign multiplicity} $k$ if the total sum of its sign of crossings with $f$ is $k$.
%Denote it as $\mu(H,f)$.
\end{definition}

\begin{definition}\label{Grassmanian-lifting}
Consider a framework $\mathcal F_C$ of an $m$-dimensional polytopal complex $C$ in $\mathbb R^n$ with a force-load $\omega$.
Let us fix some $\ell_0\in \Gr(n{-}m{-}1,n)$ that is far away (namely it does not intersect the convex hull of $\mathcal F_C$).
Let $\ell$ be an arbitrary plane in $\Gr(n{-}m{-}1,n)$.
Let now $H$ be an $\mathcal F_C$-simple path
from $\ell_0$ to $\ell$.
We say that the lifting at $\ell$ with respect to 
the path $H$ is the following number
$$
L_{\mathcal F_C,\omega,H}(\ell):=\sum\limits_{f\in F_C}
\mu(H,f) \cdot\omega_f \cdot \dist(\ell,f).
$$
\end{definition}

We conclude this paper with the following general theorem.

\begin{theorem}
The lifting $L_{\mathcal F_C,\omega,H}(\ell)$
does not depend on the choice of an
$\mathcal F_C$-simple path $H$.
\end{theorem}

\begin{proof}
{\bf The case of graph frameworks in $\mathbb R^3$.}
Here we consider the Grassmannian $\Gr(1,3)$.

Direct calculations show 
that the monodromy around a trivalent vertex is trivial.
Without loss of generality we go around the three edges $p_1p_2$, $p_1p_3$ and $p_1p_4$, where
\begin{equation*}
p_1=(0,0,0), \quad
p_2=(1,0,0), \quad
p_3=(\cos\alpha,\sin\alpha,0), \quad
p_4=(\cos\beta,\sin\beta,0).
\end{equation*}
(Here $\alpha\ne\beta$ and
$\alpha,\beta \in (0,\pi)\cup (\pi,2\pi)$). 
The starting and the ending line is $p_5p_6$ where
\begin{equation*}
p_5=(a,b,0), \quad 
p_6=(c,d,1).
\end{equation*}
In this case, the force-loads on edges $p_1p_2$, $p_1p_3$ and $p_1p_4$
are   
$\lambda\sin(\alpha-\beta)$, $\lambda\sin \beta$,
and $-\lambda\sin \alpha$, respectively (for some real $\lambda$).
Then the contribution of the monodromy around the origin (passed clockwise) is as follows:
% A reatively agly formula generated by maple
\begin{equation*}
\begin{array}{l}
\lambda\sin(\alpha-\beta)\cdot \dist(p_5p_6,p_1p_2) + 
\lambda\sin\beta\cdot \dist(p_5p_6,p_1p_3) -
\lambda\sin\alpha\cdot \dist(p_5p_6,p_1p_4))
\\
\displaystyle
\qquad=\Bigg(-{{\frac {b\sin \left( \alpha-\beta \right)}{\sqrt { \left( c-a
 \right) ^{2}+ \left( d-b \right) ^{2}+1}}}}+
{ {\frac {
   \left( a\sin \alpha -b\cos \alpha \right) \sin \beta}{\sqrt { \left( c-a \right) ^{2}+ \left( 
d-b \right) ^{2}+1}}}}
-
{ {\frac {\sin \alpha   \left( a\sin \beta  -
b \cos \beta  \right)}{\sqrt { \left( c-a
 \right) ^{2}+ \left( d-b \right) ^{2}+1}}}}\Bigg)\cdot\lambda
 \\
\qquad =0.
\end{array}
\end{equation*}

Let us now consider a vertex of general valency $d$.
Such vertex can be  decomposed locally at the vertex into the sum of trivalent vertices of the original graph
(using the techniques of quantisations introduced in~\cite{karpenkov2021frame}).
Let us illustrate it with the following example:

$$
\includegraphics[scale = .75]{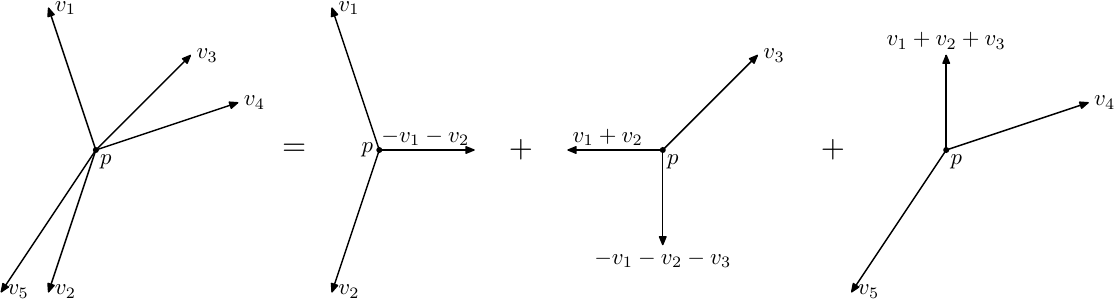}
$$

Any $\mathcal F_C$-simple closed path is decomposable into several monodromies around single vertices.
Hence, by additivity the value of the corresponding lifting is zero.

Therefore, two different $\mathcal F_C$-simple paths $H_1$ and $H_2$ to the same line $\ell$ define the same lifting (as the closed path obtained by $H_1$ and then $H_2$ inverse provides the zero lifting).
This concludes the proof for the graph frameworks in $\mathbb R^3$.

{\noindent
{\bf The case of $m$-polytopal frameworks in $\mathbb R^n$.}}
Let us sketch the proof in this case; in fact it repeats the proof of the previous theorem.

Firstly, we consider a monodromy of parallel 
$(n{-}m{-}1)$-planes around some trivalent facet $e\in E_C$. Here we assume that any line of $e$
is not parallel to a line in the parallel planes.
This monodromy can be projected to the 
monodromy of a line around a trivalent vertex of a graph
in $\mathbb R^3$. Hence it is trivial by the above case.

Secondly, the monodromy about an arbitrary facet $e$, can be decomposed into the sum of monodromies about trivalent copies of $e$.

Thirdly, by additivity the lifting along
any $\mathcal F_C$-simple closed path
is zero.

Therefore, the liftings along different paths 
provide the same value.
\end{proof}

\begin{remark}
One can restrict liftings
to a pencil of parallel planes in Grassmannians.
Then (by identifying the points of the parallel planes) we can project the 
liftings to the liftings over $\pi_0$ 
of the corresponding projected polytopal 
$m$-dimensional complex in $\mathbb R^{m+1}$.
(Here we assume that the projection of $\mathcal F_C$ does not reduce the dimensions of its faces.)
\end{remark}

\section*{Acknowledgements}
The authors are grateful to C.~Lescop for useful discussions.
F.\,M.\ was partially supported by the FWO grants G0F5921N (Odysseus) and G023721N, 
the KU Leuven iBOF/23/064 grant, and the UiT Aurora MASCOT project.
C.\,M.\ was partially supported by the Austrian Science Fund (FWF) through
grant I~4868 (SFB-Transregio “Discretization in Geometry and Dynamics”). B.S. was partially supported by the ICMS Knowledge Exchange Catalyst Programme.

%\bibliographystyle{abbrv}
%\bibliography{Differential_MC.bib}

\bigskip 

\noindent
\footnotesize \textbf{Authors' addresses:}

\bigskip

\noindent{Department of Mathematical Sciences, University of Liverpool,
%, Mathematical Sciences Building, Liverpool L69 7ZL, 
UK} 
\hfill \texttt{karpenk@liverpool.ac.uk}
\medskip  

\noindent{Departments of Mathematics and Computer Science, KU Leuven, Belgium} 
\hfill \texttt{fatemeh.mohammadi@kuleuven.be}
\medskip

\noindent{Institute of Discrete Mathematics and Geometry, TU Wien, Austria} 
\hfill \texttt{cmueller@geometrie.tuwien.ac.at}
\medskip

\noindent{Department of Mathematics and Statistics, Lancaster University, Lancaster, UK} 
\hfill \texttt{b.schulze@lancaster.ac.uk}

\end{document}